\documentclass[11pt,a4paper,reqno]{amsart}
\usepackage[utf8]{inputenc}
\usepackage[pdftex]{graphicx}
\usepackage{a4wide}
\usepackage[english]{babel}
\usepackage[T1]{fontenc}
\usepackage{amsmath,amssymb,amsthm,stmaryrd}
\usepackage{mathrsfs}
\usepackage{esint}
\usepackage{mathtools}
\usepackage{comment}
\usepackage{color}
\usepackage{extarrows}
\usepackage{tikz}
\usepackage{tcolorbox}
\usetikzlibrary{arrows}

\DeclareRobustCommand{\SkipTocEntry}[4]{}

\newcommand{\R}{\mathbb{R}}

\newcommand{\D}{\mathbb{D}}

\newcommand{\N}{\mathbb{N}}
\newcommand{\Z}{\mathbb{Z}}

\newcommand{\Bc}{\mathcal{B}}
\newcommand{\Fc}{\mathcal{F}}
\newcommand{\Cc}{\mathcal{C}}
\newcommand{\Nc}{\mathcal{N}}
\newcommand{\Sc}{\mathcal{S}}

\newcommand{\Ac}{\mathcal{A}}
\newcommand{\Hc}{\mathcal{H}}
\newcommand{\Ic}{\mathcal{I}}
\newcommand{\Gc}{\mathcal{G}}
\newcommand{\Rc}{\mathcal{R}}
\newcommand{\Pc}{\mathcal{P}}
\newcommand{\Tc}{\mathcal{T}}
\newcommand{\Wc}{\mathcal{W}}

\newcommand{\diam}{\text{diam}}
\newcommand{\dist}{\text{dist}}
\newcommand{\osc}[1]{\underset{#1}{\text{osc}} \,}

\newcommand{\eps}{\varepsilon}

\newtheorem{theorem}[equation]{Theorem}
\newtheorem*{theorem*}{Theorem}
\newtheorem{lemma}[equation]{Lemma}

\newtheorem{corollary}[equation]{Corollary}
\newtheorem{proposition}[equation]{Proposition}

\theoremstyle{definition}
\newtheorem{defin}[equation]{Definition}
\newtheorem*{defin*}{Definition}

\newtheorem*{question*}{Question}

\newtheorem{remark}[equation]{Remark}
\newtheorem*{remark*}{Remark}
\newtheorem{notation}[equation]{Notation}
\newtheorem*{notation*}{Notation}
\numberwithin{equation}{section}

\title[Uniform rectifiability and $\eps$-approximability in $L^p$]{Uniform rectifiability \\ and $\eps$-approximability of harmonic functions in $L^p$}

\author{Steve Hofmann \and Olli Tapiola}

\address{Steve Hofmann, Department of Mathematics, University of Missouri, Columbia, MO 65211, USA}
\email{hofmanns@missouri.edu}

\address{Olli Tapiola, Department of Mathematics and Statistics, P.O. Box 35 (MaD), FI-40014 University of Jyv\"askyl\"a, Finland}
\email{olli.m.tapiola@gmail.com}

\date{May 17, 2019}
\thanks{S.H. was supported by NSF grant DMS-1664047. O.T. was supported by Emil Aaltosen S\"a\"ati\"o through Foundations' Post Doc Pool grant. 
In the previous stages of this work, he was supported by the European Union through T. Hyt\"onen's ERC Starting 
Grant 278558 ``Analytic-probabilistic methods for borderline singular integrals'' and the Finnish Centre of Excellence in Analysis 
and Dynamics Research.}

\begin{document}
\begin{abstract}
  Suppose that $E \subset \R^{n+1}$ is a uniformly rectifiable set of codimension $1$. We show that every 
  harmonic function is $\eps$-approxi\-mable in $L^p(\Omega)$ for every $p \in (1,\infty)$,
  where $\Omega \coloneqq \R^{n+1} \setminus E$. Together with results of many authors this shows that 
  pointwise, $L^\infty$ and $L^p$ type $\eps$-approximability properties of harmonic functions are 
  all equivalent and they characterize uniform rectifiability for codimension $1$ Ahlfors-David regular sets.
  Our results and techniques are generalizations of recent works of T. Hyt\"onen and A. Ros\'en and the first author, J. M. Martell and S. Mayboroda.
\end{abstract}

\maketitle

\section{Introduction}

In many branches of analysis, Carleson measure estimates are powerful tools that are deeply connected to e.g. elliptic partial 
differential equations and geometric measure theory. These estimates are particularly useful for measures of the type $|\nabla u(Y)| \, dY$ (see e.g. \cite{feffermanstein, garnett}) but the problem is that even strong analytic properties of the function $u$ are not enough to guarantee that the distributional gradient 
defines a measure of this type. The idea behind \emph{$\eps$-approximability} is that although a function may fail this Carleson measure property, 
it can sometimes be approximated arbitrarily well in the $L^\infty$ sense
(typically, if it is the solution to an elliptic partial differential equation) by a function $\varphi$ such that $|\nabla \varphi(Y)| \, dY$ is a Carleson measure.
Starting from the work of N. Th. Varopoulos \cite{varopoulos} and J. Garnett \cite{garnett}, this approximation technique
has had an imporant role in the development of the theory of elliptic partial differential equations. 
It has been used to e.g. explore the absolute continuity properties of elliptic measures \cite{kenigkochpiphertoro, hofmannkenigmayborodapipher} and, 
very recently, give a new characterization of uniform rectifiability \cite{hofmannmartellmayboroda, garnettmourgogloutolsa}.

In this article, we extend the recent results of the first author, J. M. Martell and S. Mayboroda \cite{hofmannmartellmayboroda}
and show that if $E \subset \R^{n+1}$ is a uniformly rectifiable (UR) set of codimension $1$, then 
every harmonic function is $\eps$-approximable in $L^p(\Omega)$ for every $\eps \in (0,1)$ and every $p \in (1,\infty)$, where $\Omega \coloneqq \R^{n+1} \setminus E$.
The $L^p$ version of $\eps$-approximability was recently introduced by T. Hyt\"onen and A. Ros\'en \cite{hytonenrosen} who showed 
that any weak solution to certain elliptic partial differential equations in $\R^{n+1}_+$ is $\eps$-approximable in 
$L^p$ for every $\eps \in (0,1)$ and every $p \in (1,\infty)$.

Let us be more precise and recall the definition of $\eps$-approximability:
\begin{defin}
  \label{defin:eps_app}
  Suppose that $E \subset \R^{n+1}$ is an $n$-dimensional ADR set (see Definition \ref{defin:adr}) and let $\Omega \coloneqq \R^{n+1} \setminus E$
  and $\eps \in (0,1)$. We say that a function $u$ is \emph{$\eps$-approximable} if there exists a constant $C_\eps$ and 
  a function $\varphi = \varphi^\eps \in BV_{\text{loc}}(\Omega)$ satisfying
  \begin{align*}
    \|u-\varphi\|_{L^\infty(\Omega)} < \eps \ \ \ \ \ \text{ and } \ \ \ \ \ \sup_{x \in E, r > 0} \frac{1}{r^n} \iint_{B(x,r) \cap \Omega} |\nabla \varphi(Y)| \, dY \le C_\eps.
  \end{align*}
  Here $\iint_{B(x,r) \cap \Omega} |\nabla \varphi| \, dY$ stands for the total variation of $\varphi$ over $B(x,r) \cap \Omega$
  (see Section \ref{section:bounded_variation}).
\end{defin}
Sometimes $W^{1,1}$ \cite{hofmannkenigmayborodapipher} or $C^\infty$ \cite{garnett, kenigkochpiphertoro} is used 
in the definition instead of $BV_{\text{loc}}$. The first results about $\eps$-approximability showed that every bounded harmonic function $u$, normalized so that
$\|u\|_{L^\infty} \leq 1$, enjoys this this approximation property for every $\eps \in (0,1)$ in the upper half-space $\R^{n+1}_+$ 
\cite{varopoulos, garnett} and in Lipschitz domains \cite{dahlberg}. This is a highly non-trivial property since there exist
bounded harmonic functions $u$ such that $|\nabla u(Y)| \, dY$ is not a Carleson measure \cite{garnett}. The $L^p$ version 
of the property was defined only recently in \cite{hytonenrosen}:
\begin{defin}
  Suppose that $E \subset \R^{n+1}$ is an $n$-dimensional ADR set and let $\Omega \coloneqq \R^{n+1} \setminus E$, $\eps \in (0,1)$ and $p \in (1,\infty)$. 
  We say that a function $u$ is \emph{$\eps$-approximable in $L^p$} if there exists a function $\varphi = \varphi^\eps \in BV_{\text{loc}}(\Omega)$
  and constants $C_p$ and $D_{p,\eps}$ such that
  \begin{align*}
    \left\{ \begin{array}{l}
              \|N_*(u-\varphi)\|_{L^p(E)} \lesssim \eps C_p \|N_*u\|_{L^p(E)} \\
              \|\Cc(\nabla \varphi)\|_{L^p(E)} \lesssim D_{p,\eps} \|N_*u\|_{L^p(E)}
            \end{array} \right. ,
  \end{align*}
  where $N_*$ is the non-tangential maximal operator (see Definition \ref{defin:non-tangential}) and
  \begin{align*}
    \Cc(\nabla \varphi)(x) \coloneqq \sup_{r > 0} \frac{1}{r^n} \iint_{B(x,r) \cap \Omega} |\nabla \varphi| \, dY.
  \end{align*}
\end{defin}
 Here, as above, we have written $\iint_{B(x,r) \cap \Omega} |\nabla \varphi| dY$ to denote the total variation of $\varphi$ over $B(x,r) \cap \Omega$; we ask the reader to forgive this abuse of notation. See Section \ref{section:bounded_variation} for details.

In \cite{hytonenrosen}, the authors showed that if $\Omega = \R^{n+1}_+$ and $A \in L^\infty(\R^n; \mathcal{L}(\R^{n+1}))$ satisfies
$\langle A(x)v,v \rangle \ge \lambda_A |v|^2$ for almost every $x \in \R^n$ and all $v \in \R^{n+1} \setminus \{0\}$,
then any weak solution $u$ to the $t$-independent real scalar (but possibly non-symmetric) divergence form elliptic equation 
$\text{div}_{x,t} A(x) \nabla_{x,t} u(x,t) = 0$ is $\eps$-approximable in $L^p$ for any $\eps \in (0,1)$ and any $p \in (1,\infty)$.

If we move from $\R^{n+1}_+$ to the UR context (see Definition \ref{defin:ur}) with no assumptions on connectivity, 
things will not only get more complicated but we also lose many powerful tools. For example, constructing objects 
like Whitney regions and Carleson boxes becomes considerably more difficult and the harmonic measure no longer 
necessarily belongs to the class weak-$A_\infty$ with respect to the surface measure \cite{bishopjones}. 
Despite these difficulties, there exists a rich theory of harmonic analysis and many results on elliptic partial 
differential equations on sets with UR boundaries. Uniform rectifiability 
can be characterized in numerous different ways and many of these characterizations are valid in all codimensions
(see the seminal work of G. David and S. Semmes \cite{davidsemmes_singular, davidsemmes_analysis}). 
For example, UR sets are precisely those ADR sets for which certain types of singular integral operators 
are bounded from $L^2$ to $L^2$. Recently, the first author, Martell and Mayboroda showed that 
if $E$ is a UR set of codimension $1$, then every bounded harmonic function in $\R^{n+1} \setminus E$ is $\eps$-approximable 
for every $\eps \in (0,1)$ \cite{hofmannmartellmayboroda}. After this, it was shown by Garnett, Mourgoglou and Tolsa that 
$\eps$-approximability of bounded harmonic functions implies uniform rectifiability for $n$-ADR sets \cite{garnettmourgogloutolsa}. 
This characterization result was then generalized for a class of elliptic operators by Azzam, Garnett, Mourgoglou and Tolsa 
\cite{azzamgarnettmourgogloutolsa}.

Our main result is the following generalization of the Hyt\"onen-Ros\'en approximation theorem \cite[Theorem 1.3]{hytonenrosen}:
\begin{theorem}
  \label{theorem:main_result}
  Let $E \subset \R^{n+1}$ be a UR set of codimension $1$ and denote $\Omega \coloneqq \R^{n+1} \setminus E$.
  Then every harmonic function in $\Omega$ is $\eps$-approximable in $L^p$ for every $\eps \in (0,1)$ and every $p \in (1,\infty)$
  with $C_p = \|M_\D\|_{L^p \to L^p}$ and $D_p = C_p \|M\|_{L^p \to L^p}/\eps^2$, where $M$ is the Hardy-Littlewood maximal operator and $M_\D$ is its dyadic version (see Section \ref{section:notation}).
\end{theorem}
In fact, the key ideas of Hyt\"onen and Ros\'en allow us to construct $p$-independent approximating functions. To be more precise, 
let us consider the following pointwise approximating property:
\begin{defin}
  Suppose that $E \subset \R^{n+1}$ is an $n$-dimensional ADR set and let $\Omega \coloneqq \R^{n+1} \setminus E$ and $\eps \in (0,1)$. 
  We say that a function $u$ is \emph{pointwise $\eps$-approximable} if there exists a function $\varphi = \varphi^\eps \in BV_{\text{loc}}(\Omega)$
  and a constant $D_\eps$ such that
  \begin{align*}
    \left\{ \begin{array}{l}
              N_*(u-\varphi)(x) \lesssim \eps M_\D(N_*u)(x) \\
              \Cc_\D(\nabla \varphi)(x) \lesssim D_\eps M(M_\D(N_*u))(x)
            \end{array} \right. 
  \end{align*}
  for almost any $x \in E$, where $\Cc_\D$ is a dyadic version of $\D$ (see Section \ref{section:cc}).
\end{defin}
Since $\Cc(\nabla \varphi)$ and $\Cc_\D(\nabla \varphi)$ are $L^p$-equivalent by Lemma \ref{lemma:Lp-comparability_of_C}, Theorem \ref{theorem:main_result} is 
an immediate corollary of the following result and the $L^p$-boundedness of the Hardy-Littlewood maximal operator and its dyadic versions:
\begin{theorem}
  \label{thm:main_result_pointwise}
  Suppose that $E \subset \R^{n+1}$ is an $n$-dimensional UR set and let $\Omega \coloneqq \R^{n+1} \setminus E$ and $\eps \in (0,1)$. Then every harmonic function
  in $\Omega$ is pointwise $\eps$-approximable.
\end{theorem}
Although the $L^p$ version of $\eps$-approximability seems like the weakest one of all 
the properties, it is equivalent with the other properties in the codimension $1$ ADR context provided that $p$ is large enough. This follows from the recent results of S. 
Bortz and the second author \cite{bortztapiola}. Hence, combining our results with the 
results in \cite{hofmannmartellmayboroda}, \cite{garnettmourgogloutolsa} and \cite{bortztapiola} gives us the following characterization theorem:

\begin{theorem}
  Suppose that $E \subset \R^{n+1}$ is an $n$-dimensional ADR set and let $\Omega \coloneqq \R^{n+1} \setminus E$. The following conditions are equivalent:
  \begin{enumerate}
    \item[1)] $E$ is UR.

    \item[2)] Bounded harmonic functions in $\Omega$ are $\eps$-approximable for every $\eps \in (0,1)$.
    
    \item[3)] Harmonic functions in $\Omega$ are pointwise $\eps$-approximable for every $\eps \in (0,1)$.
    
    \item[4)] Harmonic functions in $\Omega$ are $\eps$-approximable in $L^p$ for some $p > n/(n-1)$ and every $\eps \in (0,1)$.
    
    \item[5)] Harmonic functions in $\Omega$ are $\eps$-approximable in $L^p$ for all $p \in (1,\infty)$ and every $\eps \in (0,1)$.
  \end{enumerate}

\end{theorem}

To prove the implication 1) $\Rightarrow$ 3), we combine some techniques of the proof of the Hyt\"onen-Ros\'en theorem with the tools and techniques 
from \cite{hofmannmartellmayboroda}. Some of the techniques can be used in a straightforward way but with the rest of them 
we have take care of many technicalities and be careful with the details.

We start by recalling the basic definitions and some results needed in our statements and proofs. For the most part, 
our notation and terminology agrees with \cite{hofmannmartellmayboroda}.

\subsection{Notation}
\label{section:notation}

We use the following notation.
\begin{enumerate}
  \item[$\bullet$] The set $E \subset \R^{n+1}$ will always be a closed set of Hausdorff dimension $n$. 
                   We denote $\Omega \coloneqq \R^{n+1} \setminus E$.
  
  \item[$\bullet$] The letters $c$ and $C$ denote constants that depend only on the dimension, the ADR constant (see Definition 
                   \ref{defin:adr}), the UR constants (see Definition \ref{defin:ur}) and other similar parameters. We call them
                   \emph{structural constants}. The values of $c$ and $C$ may change from one occurence to another. We do not track 
                   how our bounds depend on these constants and usually just write $\lambda_1 \lesssim \lambda_2$ if $\lambda_1 \le c\lambda_2$ for a structural constant
                   $c$ and $\lambda_1 \approx \lambda_2$ if $\lambda_1 \lesssim \lambda_2 \lesssim \lambda_1$.
                   
  \item[$\bullet$] We use capital letters $X,Y,Z$, and so on to denote points in $\Omega$
                   and lowecase letters $x,y,z$, and so on to denote points in $E$.
                   
  \item[$\bullet$] The $(n+1)$-dimensional Euclidean open ball of radius $r$ will be denoted $B(x,r)$ or $B(X,r)$ depending 
                   on whether the center point lies on $E$ or $\Omega$. We denote the surface ball of radius $r$ centered
                   at $x$ by $\Delta(x,r) \coloneqq B(x,r) \cap E$.
                   
  \item[$\bullet$] Given a Euclidean ball $B \coloneqq B(X,r)$ or a surface ball $\Delta \coloneqq \Delta(x,r)$ and constant 
                   $\kappa > 0$, we denote $\kappa B \coloneqq B(X,\kappa r)$ and $\kappa \Delta \coloneqq \Delta(x, \kappa r)$.
                   
  \item[$\bullet$] For every $X \in \Omega$ we set $\delta(X) \coloneqq \dist(X,E)$.
  
  \item[$\bullet$] We let $\Hc^n$ be the $n$-dimensional Hausdorff measure and denote $\sigma \coloneqq \Hc^n|_E$. The $(n+1)$-dimensional 
                   Lebesgue measure of a measurable set $A \subset \Omega$ will be denoted by $|A|$.
  
  \item[$\bullet$] For a set $A \subset \R^{n+1}$, we let $1_A$ be the indicator function of $A$: $1_A(x) = 0$ if $x \notin A$
                   and $1_A(x) = 1$ if $x \in A$.
                   
  \item[$\bullet$] The interior of a set $A$ will be denoted by $\text{int}(A)$. The closure of a set $A$ will be denoted by $\overline{A}$.
  
  \item[$\bullet$] For $\mu$-measurable sets $A$ with positive and finite measure we set $\fint_A f \, d\mu \coloneqq \tfrac{1}{\mu(A)} f \, d\mu$.
  
  \item[$\bullet$] The Hardy-Littlewood maximal operator and its dyadic version (see Section \ref{section:dyadic_cubes}) in $E$ will be denoted $M$ and $M_\D$, respectively:
                   \begin{align*}
                     Mf(x) &\coloneqq \sup_{\Delta(y,r) \ni x} \fint_{\Delta(y,r)} |f(z)| \, d\sigma(z),\\
                     M_\D f(x) &\coloneqq \sup_{Q \in \D, Q \ni x} \fint_Q |f(z)| \, d\sigma(z).
                   \end{align*}

\end{enumerate}

\subsection{ADR, UR and NTA sets}

\begin{defin}
  \label{defin:adr}
  We say that a closed set $E \subset \R^{n+1}$ is an \emph{$n$-ADR} (Ahlfors-David regular) set if there exists 
  a uniform constant $C$ such that
  \begin{align*}
    \frac{1}{C} r^n \le \sigma(\Delta(x,r)) \le C r^n
  \end{align*}
  for every $x \in E$ and every $r \in (0,\diam(E))$, where $\diam(E)$ may be infinite.
\end{defin}

\begin{defin}
  \label{defin:ur}
  Following \cite{davidsemmes_singular, davidsemmes_analysis}, we say that an $n$-ADR set $E \subset \R^{n+1}$ is \emph{UR} (uniformly rectifiable) if it contains ``big pieces of Lipschitz images'' (BPLI) of $\R^n$:
  there exist constants $\theta, \Lambda > 0$ such that for every $x \in E$ and $r \in (0,\diam(E))$ there is a Lipschitz mapping
  $\rho = \rho_{x,r} \colon \R^n \to \R^{n+1}$, with Lipschitz norm no larger that $\Lambda$, such that
  \begin{align*}
    \Hc^n(E \cap B(x,r) \cap \rho(\{y \in \R^n \colon |y| < r\})) \ge \theta r^n.
  \end{align*}
\end{defin}

\begin{defin}
  Following \cite{jerisonkenig}, we say that a domain $\Omega \subset \R^{n+1}$ is \emph{NTA} (nontangentially accessible) if
  \begin{enumerate}
    \item[$\bullet$] $\Omega$ satisfies the \emph{Harnack chain condition}: there exists a uniform constant $C$ 
                     such that for every $\rho > 0$, $\Lambda \ge 1$ and $X,X' \in \Omega$ with $\delta(X), \delta(X') \ge \rho$
                     and $|X - X'| < \Lambda \rho$ there exists a chain of open balls $B_1, \ldots, B_N \subset \Omega$, $N \le C(\Lambda)$,
                     with $X \in B_1$, $X' \in B_N$, $B_k \cap B_{k+1} \neq \emptyset$ and $C^{-1} \diam(B_k) \le \dist(B_k,\partial \Omega) \le C \diam(B_k)$,
    
    \item[$\bullet$] $\Omega$ satisfies the \emph{corkscrew condition}:
                     there exists a uniform constant $c$ such that for every surface ball $\Delta \coloneqq \Delta(x,r)$ 
                     with $x \in \partial\Omega$ and $0 < r < \diam(\partial \Omega)$ there exists a point 
                     $X_\Delta \in \Omega$ such that $B(X_\Delta, cr) \subset B(x,r) \cap \Omega$,
                     
    \item[$\bullet$] $\R^{n+1} \setminus \overline{\Omega}$ satisfies the corkscrew condition.
  \end{enumerate}
\end{defin}

\subsection{Dyadic cubes; Carleson and sparse collections}
\label{section:dyadic_cubes}

\begin{theorem}[E.g. {\cite{christ, sawyerwheeden, hytonenkairema}}]
  \label{thm:dyadic_cubes}
  Suppose that $E$ is an ADR set. Then there exists a countable collection $\D$,
  \begin{align*}
    \D \coloneqq \bigcup_{k \in \Z} \D_k, \ \ \ \ \ \D_k \coloneqq \{ Q_\alpha^k \colon \alpha \in \mathcal{A}_k \}
  \end{align*}
  of Borel sets (that we call dyadic cubes) such that
  \begin{enumerate}
    \item[$\bullet$] the collection $\D$ is \emph{nested}: $\text{if } Q,P \in \D, \text{ then } Q \cap P \in \{\emptyset,Q,P\}$,
    \item[$\bullet$] $E = \bigcup_{Q \in \D_k} Q$ for every $k \in \Z$ and the union is disjoint,
    \item[$\bullet$] there exist constants $c_1 > 0$ and $C_1 \ge 1$ with the following property: for any cube $Q_\alpha^k$ there exists a point $z_\alpha^k \in Q_\alpha^k$ (that we call the \emph{center point of $Q_\alpha^k$}) such that
                     \begin{align}
                       \label{dyadic_cubes_balls_inclusion} \Delta(z_\alpha^k,c_1 2^{-k}) \subseteq Q_\alpha^k \subseteq \Delta(z_\alpha^k, C_1 2^{-k}) \eqqcolon \Delta_{Q_\alpha^k},
                     \end{align}
    \item[$\bullet$] if $Q,P \in \D$ and $Q \subseteq P$, then
                     \begin{align}
                       \label{dyadic_cubes_big_balls_inclusion} \Delta_Q \subseteq \Delta_P,
                     \end{align}
    \item[$\bullet$] for every cube $Q_\alpha^k$ there exists a uniformly bounded number of disjoint cubes $Q_{\beta_i}^{k+1}$ such that $Q_\alpha^k = \bigcup_i Q_{\beta_i}^{k+1}$, where the uniform bound depends only on the ADR constant of $E$,
    \item[$\bullet$] the cubes form a connected tree under inclusion: if $Q, P \in \D$, then there exists a cube $R \in \D$ such that $Q \cup P \subseteq R$.
  \end{enumerate}
\end{theorem}

\begin{remark}
  The last property in the previous theorem does not appear in the constructions in \cite{christ, sawyerwheeden, hytonenkairema}, but it is 
  easy to modify the construction to get this property. The basic idea in the construction in \cite{hytonenkairema} is to choose first the center points 
  $z_\alpha^k$, then define a partial order among those points and finally build the cubes by using density arguments. 
  Thus, if we simply choose the center points $z_\alpha^k$ in such a way that there exists a point $z_0 \in \bigcap_{k \in \Z} \{z_\alpha^k\}_\alpha$, then
  by \eqref{dyadic_cubes_balls_inclusion} for any $r > 0$ there exists a cube $Q_r$ that contains the ball $B(z_0,r)$. This implies the last property in the previous theorem.
\end{remark}

\begin{notation}
  \label{notation:dyadic_cubes}
  \begin{enumerate}
    \item[1)] Since the set $E$ may be bounded or disconnected, we may encounter a situation where $Q_\alpha^k = Q_\beta^l$ although $k \neq l$. In particular, in the second to last property of Theorem \ref{thm:dyadic_cubes} there might exist only one cube $Q_{\beta_i}^{k+1}$ which equals $Q_\alpha^k$ as a set. 
              Thus, we use the notation $\D(E)$ for the collection of all relevant cubes $Q \in \D$, i.e. if $Q_\alpha^k \in \D(E)$, then 
              $C_1 2^{-k} \lesssim \diam(E)$ and the number $k$ is maximal in the sense that there does not exist a cube $Q_\beta^l \in \D$ such that 
              $Q_\beta^l = Q_\alpha^k$ for some $l > k$. Notice that the number $k$ is bounded for each cube since the ADR condition excludes the 
              presence of isolated points in $E$. This way in $\D(E)$ it is natural to talk about the children of a cube $Q$ (i.e. the largest cubes $P \subsetneq Q$) and the parent of a cube $Q$ (i.e. the smallest cube $R \supsetneq Q$).
              
    \item[2)] For every cube $Q_\alpha^k \coloneqq Q \in \D$, we denote $\ell(Q) \coloneqq 2^{-k}$ and $z_Q \coloneqq z_\alpha^k$. We call
              $\ell(Q)$ the \emph{side length} of $Q$.
              
    \item[3)] For every $Q \in \D$, we denote the collection of dyadic subcubes of $Q$ by $\D_Q$.
  \end{enumerate}
\end{notation}

\begin{defin}
  Suppose that $\Lambda \ge 1$. We say that a collection $\Ac \subset \D$ is \emph{$\Lambda$-Carleson} (or that it satisfies a \emph{Carleson packing condition}) if 
  \begin{align*}
    \sum_{Q \in \Ac, Q \subset Q_0} \sigma(Q) \le \Lambda \sigma(Q_0)
  \end{align*}
  for every cube $Q_0 \in \D$.
\end{defin}

\begin{defin}
  Suppose that $\lambda \in (0,1)$. We say that a collection $\Ac \subset \D$ is \emph{$\lambda$-sparse} if for every cube $Q \in \Ac$ 
  there exists a subset $E_Q \subset Q$ satisfying
  \begin{enumerate}
    \item[1)] $E_Q \cap E_{Q'} = \emptyset$ if $Q \neq Q'$ and
    \item[2)] $\sigma(E_Q) \ge \lambda \sigma(Q)$.
  \end{enumerate}
\end{defin}

The following result will be useful for us with some technical estimates.

\begin{theorem}
  \label{thm:sparse_carleson}
  A collection $\Ac \subset \D$ is $\Lambda$-Carleson if and only if it is $\tfrac{1}{\Lambda}$-sparse.
\end{theorem}

Although it is very easy to show that sparseness implies the Carleson property, the other implication is not obvious. For dyadic cubes in $\R^n$, it was first proven by I. Verbitsky \cite[Corollary 2]{verbitsky} and the result was later rediscovered by A. Lerner and F. Nazarov with a different proof \cite[Lemma 6.3]{lernernazarov}. For general Borel sets, the result was proven by T. H\"anninen \cite[Theorem 1.3]{hanninen}. Since the dyadic cubes in Theorem \ref{thm:dyadic_cubes} are Borel sets, the result of H\"anninen is suitable for us.

In addition to sparseness arguments, we use a discrete Carleson embedding theorem (Theorem \ref{theorem:carleson_embedding}) to prove that local bounds imply global bounds. In fact, we could use the embedding theorem instead of sparseness arguments throughout the paper but this would give us slightly weaker estimates.

\begin{defin}
  Let $\Ac \subset \D$ be any collection of dyadic cubes. We say that a cube $P \in \Ac$ is an \emph{$\Ac$-maximal subcube of $Q_0$} 
  if there do not exist any cubes $P' \in \Ac$ such that $P \subsetneq P' \subset Q_0$.
\end{defin}

\subsection{Corona decomposition, Whitney regions and Carleson boxes}
\label{subsection:whitney_regions}

\begin{defin}
  We say that a subcollection $\Sc \subset \D(E)$ is \emph{coherent} if the following three conditions hold.
  \begin{enumerate}
    \item[(a)] There exists a maximal element $Q(\Sc) \in \Sc$ such that $Q \subset \Sc$ for every $Q \in \Sc$.
    \item[(b)] If $Q \in \Sc$ and $P \in \D(E)$ is a cube such that $Q \subset P \subset Q(\Sc)$, then also $P \in \Sc$.
    \item[(c)] If $Q \in \Sc$, then either all children of $Q$ belong to $\Sc$ or none of them do.
  \end{enumerate}
  If $\Sc$ satisfies only conditions (a) and (b), then we say that $\Sc$ is \emph{semicoherent}.
\end{defin}
In this article, we do not work directly with Definition \ref{defin:ur} but use the \emph{bilateral corona decomposition} 
instead:

\begin{lemma}[{\cite[Lemma 2.2]{hofmannmartellmayboroda}}]
  Suppose that $E \subset \R^{n+1}$ is a uniformly rectifiable set of codimension $1$. Then for any pair of positive 
  constants $\eta \ll 1$ and $K \gg 1$ there exists a disjoint decomposition $\D(E) = \Gc \cup \Bc$ satisfying 
  the following properties:
  \begin{enumerate}
    \item[(1)] The ``good'' collection $\Gc$ is a disjoint union of coherent stopping time regimes $\Sc$.
    \item[(2)] The ``bad'' collection $\Bc$ and the maximal cubes $Q(\Sc)$ satisfy a Carleson packing condition:
               for every $Q \in \D(E)$ we have
               \begin{align*}
                 \sum_{Q' \subset Q, Q' \in \Bc} \sigma(Q') + \sum_{\Sc: Q(\Sc) \subset Q} \sigma(Q(\Sc)) \le C_{\eta, K} \sigma(Q).
               \end{align*}
    \item[(3)] For every $\Sc$, there exists a Lipschitz graph $\Gamma_\Sc$, with Lipschitz constant at most $\eta$, such that
               for every $Q \in \Sc$ we have
               \begin{align*}
                 \sup_{x \in \Delta_Q^*} \dist(x,\Gamma_\Sc) + \sup_{y \in B_Q^* \cap \Gamma_\Sc} \dist(y,E) < \eta \ell(Q),
               \end{align*}
               where $B_Q^* \coloneqq B(z_Q,K\ell(Q))$ and $\Delta_Q^* \coloneqq B_Q^* \cap E$.
  \end{enumerate}
\end{lemma}
The proof of this decomposition is based on the use of both the unilateral corona decomposition \cite{davidsemmes_singular}
and the bilateral weak geometric lemma \cite{davidsemmes_analysis} of David and Semmes. The decomposition plays a key role in this paper.

In \cite[Section 3]{hofmannmartellmayboroda}, the bilateral corona decomposition is used to construct Whitney regions $U_Q$ 
and Carleson boxes $T_Q$ with respect to the dyadic cubes $Q \in \D(E)$ using a dyadic Whitney decomposition of $\R^{n+1} \setminus E$. 
The Whitney regions are a substitute for the dyadic Whitney tiles $Q \times (\ell(Q)/2,\ell(Q))$ and the Carleson boxes are a substitute for the dyadic boxes 
$Q \times (0,\ell(Q))$ in $\R^{n+1}_+$. We list some of their important properties in the next lemma which we use constantly 
without specifically referring to it each time.
\begin{lemma}
  The Whitney regions $U_Q$, $Q \in \D(E)$, satisfy the following properties.
  \begin{enumerate}
    \item[$\bullet$] The region $U_Q$ is a union of a bounded number of slightly fattened Whitney cubes $I^* \coloneqq (1+\tau)I$
                     such that $\ell(Q) \approx \ell(I)$ and $\dist(Q,I) \approx \ell(Q)$. We denote the collection of these 
                     Whitney cubes by $\Wc_Q$.
    \item[$\bullet$] The regions $U_Q$ have a bounded overlap property. In particular, we have $\sum_i |U_{Q_i}| \lesssim |\bigcup_i U_{Q_i}|$ 
                     for cubes $Q_i$ such that $Q_i \neq Q_j$ if $i \neq j$.
    \item[$\bullet$] If $U_Q \cap U_P \neq \emptyset$, then $\ell(Q) \approx \ell(P)$ and $\dist(Q,P) \lesssim \ell(Q)$.
    \item[$\bullet$] For every $Y \in U_Q$ we have $\delta(Y) \approx \ell(Q)$.
    \item[$\bullet$] For every $Q \in \D(E)$, we have $|U_Q| \approx \ell(Q)^{n+1} \approx \ell(Q) \cdot \sigma(Q)$.
    \item[$\bullet$] If $Q \in \Gc$, then $U_Q$ breaks into exactly two connected components $U_Q^+$ and $U_Q^-$
                     such that $|U_Q^+| \approx |U_Q^-|$.
    \item[$\bullet$] If $Q \in \Bc$, then $U_Q$ breaks into a bounded number of connected components $U_Q^i$ such that
                     $|U_Q^i| \approx |U_Q^j|$ for all $i$ and $j$.
    \item[$\bullet$] If $\text{diam}(E) = \infty$, then $\bigcup_{Q \in \D(E)} U_Q = \Omega$.
    \item[$\bullet$] If $\text{diam}(E) < \infty$, then there exists a point $z_0 \in E$ and a constant $C \ge 1$ such that $B(z_0,C\cdot\text{diam}(E)) \setminus E \subset \bigcup_{Q \in \D(E)} U_Q$. The constant $C$ can be made large but this makes the implicit constant in the bounded overlap property large as well.
  \end{enumerate}
\end{lemma}
For every $Q \in \Gc$, the components $U_Q^+$ and $U_Q^-$ have ``center points'' that we denote by $X_Q^+$ and $X_Q^-$,
respectively. We also set $Y_Q^\pm \coloneqq X_{\widetilde{Q}}^\pm$, where $\widetilde{Q}$ is the dyadic parent of $Q$
unless $Q = Q(\Sc)$, in which case we set $\widetilde{Q} = Q$. We use these points in the construction in 
Section \ref{section:construction_local}. For any cube $Q \in \Gc$, the collection $\Wc_Q$ breaks naturally into two disjoint subcollection $\Wc_Q^+$ and $\Wc_Q^-$.

For every $Q \in \D(E)$, we define the Carleson box as the set
\begin{align*}
  T_Q \coloneqq \text{int}\left( \bigcup_{Q' \in \D_Q} U_Q \right).
\end{align*}
For each $\Ac \subset \D(E)$, we set
\begin{align}
  \label{defin:sawtooth} \Omega_\Ac \coloneqq \text{int} \left( \bigcup_{Q' \in \Ac} U_{Q'} \right).
\end{align}

\subsection{Local $BV$}
\label{section:bounded_variation}

\begin{defin}
  We say that a function $f \in L^1_\text{loc}(\Omega)$ has \emph{locally bounded variation} (denote $f \in BV_{\text{loc}}(\Omega)$) if 
  for any bounded open set $U \subset \Omega$ such that $\overline{U} \subset \Omega$ we have
  \begin{align*}
    \sup_{\substack{\overrightarrow{\Psi} \in C_0^1(U),\\ \|\overrightarrow{\Psi}\|_{L^\infty} \le 1}} \iint_{U} f(Y) \text{div}\overrightarrow{\Psi}(Y) \, dY < \infty.
  \end{align*}
\end{defin}
The latter expression can be shown to define a measure, by the Riesz representation theorem.  We have the following:
\begin{theorem}[{\cite[Section 5.1]{evansgariepy}}]
  Suppose that $f \in BV_\text{loc}(\Omega)$. Then there exists a Radon measure $\mu$ on $\Omega$ such that
  \begin{align*}
    \mu(U) =    \sup_{\substack{\overrightarrow{\Psi} \in C_0^1(U),\\ \|\overrightarrow{\Psi}\|_{L^\infty} \le 1}} \iint_{U} f(Y) \text{div}\overrightarrow{\Psi}(Y) \, dY. 
  \end{align*}
  for any open set $U \subset \Omega$; 
  we call
  $\mu(U)$  the {\em total variation} of $f$ on $U$.
\end{theorem}
Abusing notation, for  an open set $U\subset \Omega$, we shall write
\[\mu(U):= \iint_U |\nabla f(Y)| \, dY,\] 
which should not be mistaken for a usual Lebesgue integral. 
Indeed, we may have situations where $A \subset B$ and $|A| = |B|$ but
$\iint_A |\nabla f(Y)| \, dY \ll \iint_B |\nabla f(Y)| \, dY$. 

In particular, if $f \in BV_\text{loc}(\Omega)$, the sets $U, U_1, \ldots, U_k \subset \Omega$ are open and $U \subset \bigcup_i U_i$, then
\begin{align}
  \label{estimate:additive_variation} \iint_U |\nabla f(Y)| \, dY \le \sum_i \iint_{U_i} |\nabla f(Y)| \, dY.
\end{align}

\begin{remark}   We emphasize that we write $ |\nabla f| dY$ to indicate the variation measure of $f$, which is denoted
by $\|Df\|$ in \cite{evansgariepy};  thus, for $f\in BV_{loc}(\Omega)$, and for any open set $U\subset \Omega$, we let
$\iint_{U} |\nabla f| dY$  denote the total variation of $f$ over $U$. We shall continue to use this 
(mildly abusive) notational convention in the sequel,
when working with elements of $BV_{loc}(\Omega)$.
\end{remark}

\subsection{$\Cc$ and $\Cc_\D$}
\label{section:cc}
For every $k \in \N$, we let $F_k$ be the ordered pair $(E,k)$. In this section, we let $Q_0 = E$ be the maximal dyadic cube if 
$E$ is a bounded set. We define the operators $\Cc$ and $\Cc_\D$ by setting
\begin{align*}
  \Cc(f)(x) &\coloneqq \sup_{r > 0} \frac{1}{r^n} \iint_{B(x,r) \setminus E} |f(Y)| \, dY, \\
  \Cc_\D(f)(x) &\coloneqq \sup_{Q \in \D^*, x \in Q} \frac{1}{\ell(Q)^n} \iint_{T_Q} |f(Y)| \, dY,
\end{align*}
where
\begin{align*}
  \D^* \coloneqq \left\{ \begin{array}{cl}
                           \D(E), &\text{if } \ \diam(E) = \infty \\
                           \D(E) \cup \{F_k \colon k=\Lambda_0,\Lambda_0+1,\ldots\}, &\text{if } \ \diam(E) < \infty
                         \end{array} \right.
\end{align*}
and
\begin{align*}
  T_{F_k} \coloneqq B(z_0, 2^k \diam(E)), \ \ \ \ \ \ \ell(F_k) \coloneqq 2^k \diam(E)
\end{align*}
for some fixed point $z_0 \in E$ and a number $\Lambda_0$ such that $T_{Q_0} \subset T_{F_{\Lambda_0}}$. We will call also the pairs $F_k$ cubes although their actual 
structure is irrelevant and we will interpret $x \in F_k$ simply as $x \in E$.

Usually, these functions are not pointwise equivalent but we only have $\Cc_\D(f)(x) \lesssim \Cc(f)(x)$ for every $x \in E$
(this follows from the ADR property of $E$ and the fact that $T_Q \subset B(z_Q, C\ell(Q))$ for a uniform constant $C$). However, 
in $L^p$ sense, these functions are always comparable. This can be seen easily from the level set comparison formula that 
we prove next. This comparability is convenient for us since we construct the approximating function $\varphi$ in Theorem \ref{theorem:main_result}
with the help of the dyadic Whitney regions. Thus, it is more natural for us to prove the desired $L^p$ bound for 
$\Cc_\D(\nabla \varphi)$ instead of $\Cc(\nabla \varphi)$. We prove the comparison formula by using well-known techniques from the proof of the 
corresponding formula for the Hardy-Littlewood maximal function and its dyadic version \cite[Lemma 2.12]{duoandikoetxea}.

\begin{lemma}
  \label{lemma:Lp-comparability_of_C}
  Suppose that $f \in BV_\text{loc}(\Omega)$. Then there exist uniform constants $A_1$ and $A_2$ (depending on the dimension and the ADR constant) such that for every $\lambda > 0$ we have
  \begin{align*}
    \sigma\left(\left\{ x \in E \colon \Cc(\nabla f)(x) > A_1 \lambda \right\}\right) \ \le  \ A_2 \cdot \sigma\left(\left\{ x \in E \colon \Cc_\D(\nabla f)(x) > \lambda \right\} \right).
  \end{align*}
  In particular, $\| \Cc(f) \|_{L^p(E)} \le A_1 A_2^{1/p} \| \Cc_\D(f) \|_{L^p(E)}$ for every $p \in (1,\infty)$.
\end{lemma}

\begin{proof}
  We first note that if $r \gg \diam(E)$, then by the definition of $\Cc_\D$ we have the bound 
  $\tfrac{1}{r^n} \iint_{B(x,r) \setminus E} |\nabla f(Y)| \, dY \lesssim \Cc_\D(\nabla f)(x)$. Thus,
  we may assume that the balls in this proof have uniformly bounded radii $\lesssim \diam(E)$ and the cubes belong to $\D(E)$.
  Naturally, we may also assume that the right hand side of the inequality is finite.
  
  We notice that if $\Cc_\D(f)(x) > \lambda$, then there exists a cube $Q \in \D(E)$ such that $x \in Q$ and $\tfrac{1}{\sigma(Q)} \iint_{T_Q} |\nabla f(Y)| \, dY > \lambda$. By the definition of $\Cc_\D(f)$, we also have $\Cc_\D(f)(y) > \lambda$ for every $y \in Q$. In particular, we have
  \begin{align*}
    \left\{ x \in E \colon \Cc_\D(\nabla f)(x) > \lambda \right\} = \bigcup_i Q_i
  \end{align*}
  for disjoint dyadic cubes $Q_i$. We now claim that if $A_1$ is large enough, then
  \begin{align}
    \label{inclusion_level_sets} \left\{ x \in E \colon \Cc(\nabla f)(x) > A_1 \lambda \right\} \subseteq \bigcup_i 2\Delta_{Q_i}
  \end{align}
  where $\Delta_{Q_i}$ is the surface ball \eqref{dyadic_cubes_balls_inclusion}. Suppose that $y \notin \bigcup_i 2\Delta_{Q_i}$ and let $r > 0$. 
  Let us choose $k \in \Z$ so that $2^{k-1} \le r < 2^k$. Now there exist at most $K$ dyadic cubes $R_1, R_2, \ldots, R_m$ such 
  that $\ell(R_j) = 2^k$ and $R_j \cap \Delta(y,r) \neq \emptyset$ for every $j=1,2,\ldots,m$. We notice that none of the cubes $R_j$ can be contained in any of the cubes 
  $Q_i$ since otherwise we would have $y \in 2\Delta_{R_j} \subset 2\Delta_{Q_i}$ by \eqref{dyadic_cubes_big_balls_inclusion}.
  Thus, we have $\tfrac{1}{\ell(R_j)^n} \iint_{T_{R_j}} |\nabla f(Y)| \, dY \le \lambda$ for every $j$.
  We can use a straightforward geometric argument to show that $B(y,r) \subset \bigcup_{j=1}^m T_{R_j}$ 
  (see \cite[pages 2353-2354]{hofmannmartellmayboroda}).
  Hence, since $r \approx \ell(R_j)$ for every $j$, we have
  \begin{align*}
    \frac{1}{r^n} \iint_{B(y,r)} |\nabla f(Y)| \, dY &\overset{\eqref{estimate:additive_variation}}{\lesssim} \sum_{j=1}^m \frac{1}{\ell(R_j)^n} \iint_{T_{R_j}} |\nabla f(Y)| \, dY
                                                     \lesssim \lambda
  \end{align*}
  and $y \notin \left\{ x \in E \colon \Cc(\nabla f)(x) > A_1 \lambda \right\}$ for a large enough $A_1$. In particular, \eqref{inclusion_level_sets} holds 
  and we have
  \begin{align*}
    \sigma( \left\{ x \in E \colon \Cc(\nabla f)(x) > A_1 \lambda \right\})
    &\le \sum_i \sigma(2\Delta_{Q_i}) \\
    &\lesssim \sum_i \sigma(Q_i) \\
    &= \sigma\left( \bigcup_i Q_i \right)
    = \sigma( \left\{ x \in E \colon \Cc_\D(\nabla f)(x) > \lambda \right\}).
  \end{align*}
  The $L^p$ comparability $\Cc(\nabla f)$ and $\Cc_\D(\nabla f)$ follows immediately:
  \begin{align*}
    \|\Cc(\nabla f)\|_{L^p(E)}^p &= p \int_0^\infty \lambda^{p-1} \sigma(\{x \in E \colon \Cc(\nabla f)(x) > \lambda\}) \, d\lambda \\
                                 &\le A_2 p \int_0^\infty \lambda^{p-1} \sigma(\{x \in E \colon A_1 \Cc_\D(\nabla f)(x) > \lambda\}) \, d\lambda \\
                                &= A_1^p A_2 \|\Cc_D(\nabla f)\|_{L^p(E)}^p.
  \end{align*}

\end{proof}

\subsection{Cones, non-tangential maximal functions and square functions}
\label{subsection:non-tangential}

We recall from \cite[Section 3]{hofmannmartellmayboroda} that the Whitney regions $U_Q$ and the fattened Whitney regions $\widehat{U}_Q$, $Q \in \D$, are defined using fattened Whitney boxes $I^* \coloneqq (1 + \tau)I$ and $I^{**} \coloneqq (1 + 2\tau)I$ respectively, where $\tau$ is a suitable positive parameter. Let us define the regions $\widehat{\textbf{U}}_Q$ using even fatter Whitney boxes $I^{***} \coloneqq (1 + 3\tau)W$.

\begin{defin}
  \label{defin:non-tangential}
  For any $x \in E$, we define the \emph{cone at $x$} by setting
  \begin{align}
    \label{defin:cone} \Gamma(x) \coloneqq \bigcup_{Q \in \D(E), Q \ni x} \widehat{\textbf{U}}_Q.
  \end{align}
  We define the \emph{non-tangential maximal function} $N_*u$ and, for $u \in W^{1,2}_{\text{loc}}(\Omega)$, 
  the square function $Su$ as follows:
  \begin{align*}
    N_*u(x) &\coloneqq \sup_{Y \in \Gamma(x)} |u(Y)|, \ \ \ \ \ x \in E,\\
    Su(x)   &\coloneqq \left( \int_{\Gamma(x)} |\nabla u(Y)|^2 \delta(Y)^{1-n} \, dY \right)^{1/2}, \ \ \ \ \ x \in E.
  \end{align*}
\end{defin}
The Hyt\"onen-Ros\'en techniques in \cite[Section 6]{hytonenrosen} rely on the use of local $S \lesssim N$ and $N \lesssim S$ estimates 
from \cite{hofmannkenigmayborodapipher}. Although a local $S \lesssim N$ estimate holds also in our context \cite{hofmannmartellmayboroda_square}, a local $N \lesssim S$ estimate 
does not hold without suitable assumptions on connectivity. Thus, we cannot apply the Hyt\"onen-Ros\'en techniques directly but we have to combine them with the techniques created in 
\cite{hofmannmartellmayboroda}.

In Section \ref{section:construction} we consider the following modified versions of $\Gamma(x)$ and $N_* u$ to bypass some additional technicalities:

\begin{defin}
  For every $x \in E$ and $\alpha > 0$ we define the \emph{cone of $\alpha$-aperture at $x$} $\Gamma_\alpha(x)$ by setting
  \begin{align}
    \label{defin:modified_cone} \Gamma_\alpha(x) \coloneqq \bigcup_{Q \in \D(E), Q \ni x} \bigcup_{\substack{P \in \D(E), \\ \ell(P) = \ell(Q), \\ \alpha \Delta_Q \cap P \neq \emptyset}} \widehat{\textbf{U}}_P.
  \end{align}
  Using the cones $\Gamma_\alpha(x)$, we define the \emph{non-tangential maximal function of $\alpha$-aperture} $N^\alpha_* u$
  by setting $N_*^\alpha u(x) \coloneqq \sup_{Y \in \Gamma_\alpha(x)} |u(Y)|$.
\end{defin}

\begin{remark}
  If the set $E$ is bounded, then the cones \eqref{defin:cone} and \eqref{defin:modified_cone} are also bounded since we 
  only constructed Whitney regions $U$ such that $\diam(U) \lesssim \diam(E)$. Thus, if $E$ is bounded, we use the cones
  \begin{align*}
    &\widehat{\Gamma}(x) \coloneqq \Gamma(x) \cup B(z_0, C\cdot\diam(E))^c \text{ and} \\
    &\widehat{\Gamma}_\alpha(x) \coloneqq \Gamma_\alpha(x) \cup B(z_0, C_\alpha\cdot\diam(E))^c
  \end{align*}
  for a suitable point $z_0 \in E$ and suitable constants $C$ and $C_\alpha$ instead.
\end{remark}
The usefulness of these modified cones and non-tangential maximal functions lies in the fact that for a suitable choice of $\alpha$ the cone $\Gamma_\alpha(x)$ 
contains some crucial points that may not be contained in $\Gamma(x)$ and in the $L^p$ sense the function $N_*^\alpha u$ is not too much larger than 
$N_* u$. We prove the latter claim in the next lemma but postpone the proof of the first claim to Section \ref{section:construction}.

\begin{lemma}
  \label{lemma:Lp-comparability_of_N}
  Suppose that $u$ is a continuous function and let $\alpha \ge 1$. Then $\|N_*u\|_{L^p(E)} \approx_\alpha \|N_*^\alpha u\|_{L^p(E)}$
  for every $p \in (0,\infty)$.
\end{lemma}

\begin{proof}
  We only prove the claim for the case $\diam(E) = \infty$ as the proof for the case $\diam(E) < \infty$ is almost the same.
  
  Since the set $E$ is ADR, measures of balls with comparable radii are comparable. Using this property makes it is simple 
  and straightforward to generalize the classical proof of C. Fefferman and E. Stein \cite[Lemma 1]{feffermanstein} from $\R^{n+1}_+$ to $\Omega$
  to show that $\| \Nc_\alpha u\|_{L^p(E)} \approx_{\alpha,\beta} \| \Nc_\beta u\|_{L^p(E)}$ where
  \begin{align*}
    \Nc_\gamma u(x) \coloneqq \sup_{Y \in \widetilde{\Gamma}_\gamma(x)} |u(Y)|, \ \ \ \ \ 
    \widetilde{\Gamma}_\gamma(x) \coloneqq \left\{ Y \in \Omega \colon \dist(x,Y) < \gamma \cdot \delta(Y) \right\}.
  \end{align*}
  By the definition of the cones $\Gamma(x)$, there exists $\gamma_0 > 0$ such that 
  $\widetilde{\Gamma}_{\gamma_0}(x) \subset \Gamma(x)$ for every $x \in E$. Thus, we only need to show that
  $\Gamma_\alpha(x) \subset \widetilde{\Gamma}_\gamma(x)$ for some uniform $\gamma = \gamma(\alpha)$ for all $x \in E$ since this gives us the estimate (*) 
  in the chain
  \begin{align*}
    \|N_* u\|_{L^p(E)}
    \le \|N_*^\alpha u\|_{L^p(E)}
    &\overset{\text{(*)}}{\le} \|\Nc_{\gamma} u\|_{L^p(E)} \\
    &\approx_{\gamma,\gamma_0} \|\Nc_{\gamma_0} u\|_{L^p(E)}
    \le \|N_* u\|_{L^p(E)}.
  \end{align*}
  Suppose that $Q, P \in \D(E)$, $x \in Q$, $\ell(Q) = \ell(P)$ and $\alpha \Delta_Q \cap P \neq \emptyset$. By the construction 
  of the Whitney regions, for every $Y \in \widehat{\textbf{U}}_P$ we have
  \begin{align*}
    \delta(Y) \approx \ell(P) \approx \dist(Y,P).
  \end{align*}
  On the other hand, since $\alpha \Delta_Q \cap P \neq \emptyset$ and $\ell(P) = \ell(Q)$, we know that for any $y \in P$ we have
  \begin{align*}
    \dist(x,y) \lesssim \alpha \ell(Q) = \alpha \ell(P).
  \end{align*}
  Let us take any $z \in P$. Now for every $Y \in \widehat{\textbf{U}}_P$ we have
  \begin{align*}
    \dist(x,Y) \le \dist(x,z) + \dist(z,Y) \lesssim \alpha \ell(P) + \ell(P) \lesssim \alpha \ell(P) \approx \alpha \cdot \delta(Y).
  \end{align*}
  In particular, there exists a uniform constant $\gamma = \gamma(\alpha)$ such that $\Gamma_\alpha(x) \subset \widetilde{\Gamma}_{\gamma}(x)$.
\end{proof}

\section{Principal cubes}

As in \cite{hytonenrosen}, we define the numbers $M_\D(N_*u)(Q)$ by setting
\begin{align*}
  M_{\D}(N_* u)(Q) \coloneqq \sup_{Q \subseteq R \in \D} \fint_R N_* u(y) \, d\sigma(y)
\end{align*}
for every $Q \in \D(E) \eqqcolon \D$. We shall use a collection $\Ic \subset \D(E) = \D$ such that
\begin{align}
  \label{initial_collection} \Ic \coloneqq \left\{Q_i \colon i \in \widetilde{\N} \right\}, \ \ \ \ \ Q_i \subsetneq Q_{i+1} \ \forall i, \ \ \ \ \ \bigcup_i Q_i = E,
\end{align}
where $\widetilde{\N} = \{1,2,\ldots,n_0\}$ for some $n_0 \in \N$ if $E$ is bounded, and $\widetilde{\N} = \N$ otherwise. 
This type of a collection exists by the last property in Theorem \ref{thm:dyadic_cubes} and by the properties of dyadic cubes, the 
collection is Carleson. Let us construct a collection $\Pc \subset \D$ of ''stopping cubes`` using the construction described in
\cite[Section 6.1]{hytonenrosen}. We set $\Pc_0 \coloneqq \Ic$ and consider all the cubes $Q' \in \D(E) \setminus \Pc_0$ such that
\begin{enumerate}
  \item[(a)] for some $Q \in \Pc_0$ we have $Q' \subsetneq Q$ and
             \begin{align}
               \label{stopping_condition_1} M_{\D}(N_* u)(Q') = \sup_{Q' \subseteq R \in \D} \fint_R N_* u(y) \, d\sigma(y) > 2M_{\D}(N_* u)(Q),
             \end{align}
  \item[(b)] $Q'$ is not contained in any such $Q'' \subsetneq Q$ such that either $Q'' \in \Pc_0$ or \eqref{stopping_condition_1} holds for the pair
             $(Q'',Q)$.
\end{enumerate}
We denote by $\Pc_1$ the collection we get by adding all the cubes $Q'$ satisfying both (a) and (b) to $\Pc_0$. We then continue this process for $\Pc_1$ in place of $\Pc_0$ and 
so on. We set $\Pc \coloneqq \bigcup_{k=0}^\infty \Pc_k$. We also set
\begin{align*}
  \pi_\Pc Q = \text{ the smallest cube } Q_0 \in \Pc \text{ such that } Q \subseteq Q_0.
\end{align*}
Here we mean smallest with respect to the side length. Naturally, we have $\pi_\Pc Q = Q$ for every $Q \in \Pc$, and since $\Ic \subset \Pc$, for every cube $Q \in \D$ there exists some cube $P_Q \in \Pc$ such that $Q \subset P_Q$.

\begin{remark}
  \label{remark:simplification}
  The collection $\Pc$ is an auxiliary collection that helps us to simplify the proofs of several claims. We use it in the following way. Suppose 
  that we have a subcollection $\mathcal{W} \subset \D$ and we want to show that $\mathcal{W}$ satisfies a Carleson packing condition. 
  Let $Q_0 \in \D$. Now for every $Q \in \mathcal{W}$ such that $Q \subset Q_0$, we have either $\pi_\Pc Q = \pi_\Pc Q_0$ or $\pi_\Pc Q = P = \pi_\Pc P$
  for some $P \in \Pc$ such that $P \subsetneq \pi_\Pc Q_0$. In particular, we have
  \begin{align*}
    \sum_{Q \in \mathcal{W}, Q \subseteq Q_0} \sigma(Q)
    &= \sum_{\substack{Q \in \mathcal{W}, \\ \pi_\Pc Q = \pi_\Pc Q_0}} \sigma(Q) + \sum_{P \in \Pc, P \subsetneq \pi_\Pc Q_0} \sum_{\substack{Q \in \mathcal{W},\\ \pi_\Pc Q = P}} \sigma(Q)
    \eqqcolon I_{Q_0} + \sum_{P \in \Pc, P \subsetneq \pi_\Pc Q_0} I_P.
  \end{align*}
  We prove in Lemma \ref{lemma:p-carleson} below that the collection $\Pc$ satisfies a Carleson packing condition. Thus, if we can show that $I_{Q_0} \lesssim \sigma(Q_0)$
  for an arbitrary cube $Q_0 \in \Pc$, we get
  \begin{align*}
    \sum_{P \in \Pc, P \subsetneq \pi_\Pc Q_0} I_P \lesssim \sum_{P \in \Pc, P \subsetneq \pi_\Pc Q_0} \sigma(P) \lesssim \sigma(Q_0).
  \end{align*}
  Thus, to show that the collection $\mathcal{W}$ satisfies a Carleson packing condition, it is enough to show that $I_{Q_0} \lesssim \sigma(Q_0)$ 
  for every cube $Q_0 \in \D$. The usefulness of this simplification is that if $Q \in \D \setminus \Pc$ and $\pi_\Pc Q = P$, then 
  by the construction of the collection $\Pc$ we have
  \begin{align*}
    M_{\D}(N_* u)(Q) \le 2M_{\D}(N_* u)(P).
  \end{align*}
  We use this property several times in the proofs.
\end{remark}

For any cube $Q_0 \in \D$, we say that $R \in \Pc$ is a \emph{$\Pc$-proper subcube of $Q_0$} if we have
$M_{\D}(N_* u)(R) > 2M_{\D}(N_* u)(Q_0)$ and $M_{\D}(N_* u)(R') \le 2M_{\D}(N_* u)(Q_0)$ for every intermediate cube $R \subsetneq R' \subsetneq Q_0$. 

\begin{lemma}
  \label{lemma:p-carleson}
  For every $Q_0 \in \D(E)$ we have
  \begin{align}
    \label{estimate:Pc_carleson} \sum_{P \in \Pc, P \subseteq Q_0} \sigma(P) \lesssim \sigma(Q_0).
  \end{align}
\end{lemma}

\begin{proof}
  Let us start by noting that we may assume that $Q_0 \in \Pc$ since otherwise we can simply consider the $\Pc$-maximal subcubes of $Q_0$. To be more precise, the $\Pc$-maximal subcubes of $Q_0$ are disjoint by definition and thus, if we sum their measures together, it is at most $\sigma(Q_0)$. Now, if $Q \in \Pc$ and $Q \subset Q_0$, we know that $Q$ is one of the $\Pc$-maximal subsubes of $Q_0$ or it is contained properly in one of them. Hence, if we prove the estimate \eqref{estimate:Pc_carleson} for the case $Q_0 \in \Pc$, it implies the same estimate even with the same implicit constant for the case $Q_0 \notin \Pc$.

  Suppose first that we have a collection of disjoint cubes $Q' \subset Q$ that satisfy $M_{\D}(N_* u)(Q') > 2M_{\D}(N_* u)(Q)$. Then, for every such 
  cube $Q'$ we have $M_{\D}(N_* u)(Q') > \fint_Q N_*u \, d\sigma$ and thus, for every point $x \in Q'$ we get
  \begin{align*}
    M_{\D}(1_Q N_*u)(x) &= \sup_{R \in \D, x \in R \subseteq Q} \fint_R N_*u \, d\sigma \\
                        &\ge \sup_{R \in \D, Q' \subseteq R \subsetneq Q} \fint_R N_*u \, d\sigma
                        = M_{\D}(N_*u)(Q')
                        > 2M_{\D}(N_*u)(Q).
  \end{align*}
  In particular, by the $L^1 \to L^{1,\infty}$ boundedness of $M_\D$ we have
  \begin{align}
    \nonumber \sum_{Q'} \sigma(Q') &\le \sigma\left(\left\{x \in E \colon M_{\D}(1_Q N_*u)(x) > 2M_{\D}(N_*u)(Q) \right\}\right) \\
    \label{apu1}                   &\le \frac{1}{ 2M_{\D}(N_*u)(Q)} \|1_Q N_*u\|_{L^1(\sigma)}
                                    = \frac{\fint_Q N_*u \, d\sigma}{M(N_* u)(Q)} \frac{\sigma(Q)}{2} \le \frac{\sigma(Q)}{2}.
  \end{align}
  
  We notice that if $R \in \Pc \setminus \Ic$, then $R$ is a $\Pc$-proper subcube of some cube $Q \in \Pc$. To be more precise, 
  if $R \in \Pc \setminus \Ic$, then there exists a chain of cubes 
  $R = R_1 \subsetneq R_2 \subsetneq \ldots \subsetneq R_k$, $R_i \in \Pc$, such that for every $i = 1,2,\ldots,k-1$ 
  $R_i$ is a $\Pc$-proper subcube of $R_{i+1}$ and $R_k \in \Ic$. If such a chain of length $k$ from $R$ to $Q$ exists, 
  we denote $R \in \Pc_Q^k$. By using the property \eqref{apu1} $k$ times, we see that for each $Q \in \Pc$ we have
  \begin{align}
    \label{apu2} \sum_{R \in \Pc_Q^k} \sigma(R) \le \sum_{R \in \Pc_Q^{k-1}} \sum_{S \in \Pc_Q^k, S \subsetneq R} \sigma(S)
                                                \le \frac{1}{2} \sum_{R \in \Pc_Q^{k-1}} \sigma(R) 
                                                \le \ldots
                                \le \frac{1}{2^{k-1}} \sum_{R \in \Pc_Q^1} \sigma(R) \le \frac{\sigma(Q)}{2^k}.
  \end{align}
  Now it is straightforward to prove the packing condition. We have
  \begin{align*}
    \sum_{P \in \Pc, P \subseteq Q_0} \sigma(P) \ &= \ \sum_{P \in \Ic, P \subseteq Q_0} \sigma(P) + \sum_{P \in \Pc \setminus \Ic, P \subseteq Q_0} \sigma(P) \\
                                                \ &\le \ C_\Ic \sigma(Q_0) + \sum_{Q \in \Ic, Q \subseteq Q_0} \sum_{k=1}^\infty \sum_{P \in \Pc_Q^k} \sigma(P) \\
                                                \ &\overset{\eqref{apu2}}{\le} \ C_\Ic \sigma(Q_0) + \sum_{Q \in \Ic, Q \subseteq Q_0} \sum_{k=1}^\infty \frac{\sigma(Q)}{2^k} \\
                                                \ &= C_\Ic \sigma(Q_0) + \sum_{Q \in \Ic, Q \subseteq Q_0} \sigma(Q) \\
                                                &\le C_\Ic \sigma(Q_0) + C_\Ic \sigma(Q_0)
  \end{align*}
  which proves the claim.
\end{proof}

\section{``Large Oscillation'' cubes}

Before constructing the approximating function, we consider two collections of cubes that will act as the basis of our construction. In this section, we show that 
the union of the collection of ``large oscillation'' cubes
\begin{align*}
  \mathcal{R} \coloneqq \left\{Q \in \D \colon \osc{U_Q^i} u > \varepsilon M_\D(N_*u)(Q) \text{ for some } i\right\}.
\end{align*}
and the collection of ``bad'' cubes from the corona decomposition satisfies a Carleson packing condition. We apply this property in the technical estimates in Section \ref{section:construction}.

\begin{lemma}
  \label{lemma:oscillation_carleson}
  For every $Q_0 \in \D(E)$ we have
  \begin{align}
    \label{estimate:oscillation_carleson} \sum_{R \in \Rc, R \subseteq Q_0} \sigma(R) \lesssim \frac{1}{\varepsilon^2} \sigma(Q_0).
  \end{align}
\end{lemma}

\begin{proof}
We break the proof into three parts.

\textbf{Part 1: Simplification.} First, by Remark \ref{remark:simplification}, it is enough to show that
\begin{align*}
  \underset{\pi_\Pc R = \pi_\Pc Q_0}{\sum_{R \in \Rc, R \subset Q_0}} \sigma(R) \lesssim \frac{1}{\eps^2} \sigma(Q_0).
\end{align*}
Also, since the ``bad'' collection in the bilateral corona decomposition is Carleson, it suffices to consider the ``good'' cubes in $\Rc$, i.e. the collection $\Rc \cap \mathcal{G}$.
Thus, we may assume that $Q_0 \in \Rc \cap \mathcal{G}$ since otherwise we may simply consider the $(\Rc \cap \mathcal{G})$-maximal subcubes of $Q_0$ similarly as with the collection $\Pc$ in the proof of Lemma \ref{lemma:p-carleson}. Furthermore, since the Whitney regions $U_R$ of the ``good'' cubes $R$ break into two components $U_R^+$ and $U_R^-$, it is enough to 
bound the sum
\begin{align*}
  \underset{\pi_\Pc R = \pi_\Pc Q_0}{\sum_{R \in \Rc^+, R \subset Q_0}} \sigma(R) \lesssim \sigma(Q_0),
\end{align*}
where $\Rc^+ \coloneqq \{Q \in \Rc \cap \mathcal{G} \colon \text{osc}_{U_Q^+} > \varepsilon M_\D(N_*u)(Q) \}$, as the arguments for the corresponding collection $\Rc^-$ are the same.

Since $Q_0 \in \Gc$, there exists a stopping time regime $\Sc_0 = \Sc_0(Q_0)$ such that $Q_0 \in \Sc_0$. We note that if we have $Q \subset Q_0$ for a cube $Q \in \Rc^+$, then either $Q \in \Sc_0$ or,
by the coherency and disjointness of the stopping time regimes, $Q_0 \in \Sc$ for such a $\Sc$ that $Q(\Sc) \subsetneq Q_0$. Let $\mathfrak{S} = \mathfrak{S}(Q_0)$ be the 
collection of the stopping time regimes $\Sc$ such that $Q(\Sc) \subsetneq Q_0$. Then we have
\begin{align*}
  \underset{\pi_\Pc R = \pi_\Pc Q_0}{\sum_{R \in \Rc^+, R \subset Q_0}} \sigma(R) 
  &= \underset{\pi_\Pc R = \pi_\Pc Q_0}{\sum_{R \in \Rc^+ \cap \Sc_0, R \subset Q_0}} \sigma(R) +
    \sum_{\Sc \in \mathfrak{S}} \underset{\pi_\Pc R = \pi_\Pc Q_0}{\sum_{R \in \Rc^+ \cap \Sc, R \subset Q_0}} \sigma(R) \\
  &\eqqcolon I_{Q_0} + II_{Q_0}.
\end{align*}
Let us show that if $I_{Q_0} \lesssim \sigma(Q_0)$ for every $Q_0 \in \D$, then $II_{Q_0} \lesssim \sigma(Q_0)$ for every $Q_0 \in \D$. 
Suppose that $Q \in \Sc \in \mathfrak{S}$. Since $Q(\Sc) \subsetneq Q_0$, we 
have $\pi_\Pc Q = \pi_\Pc Q_0$ only if $\pi_\Pc Q = \pi_\Pc Q(S) = \pi_\Pc Q_0$. Thus, it holds that
\begin{align*}
  II_{Q_0} = \sum_{\Sc \in \mathfrak{S}} \underset{\pi_\Pc R = \pi_\Pc Q_0}{\sum_{R \in \Rc^+ \cap \Sc, R \subset Q_0}} \sigma(R) 
           \le \sum_{\Sc \in \mathfrak{S}} \underset{\pi_\Pc R = \pi_\Pc Q(\Sc)}{\sum_{R \in \Rc^+ \cap \Sc, R \subset Q_0}} \sigma(R) 
           = \sum_{\Sc \in \mathfrak{S}} I_{Q(\Sc)} 
           \lesssim \sum_{\Sc \in \mathfrak{S}} \sigma(Q(\Sc)) 
           \lesssim \sigma(Q_0)
\end{align*}
by the Carleson packing property of the collection $\{Q(\Sc)\}_\Sc$. Hence, to prove \eqref{estimate:oscillation_carleson}, it suffices to show $I_{Q_0} \lesssim \sigma(Q_0)$.

\

\textbf{Part 2: $\delta(Y) \lesssim D_\mathcal{A}(Y)$ in $\widehat{U}_P^+$.} Let $\mathcal{A} \subset \Gc$ be a collection of cubes and set
\begin{align*}
  \Omega_\mathcal{A}^* \coloneqq \text{int} \left( \bigcup_{Q \in \mathcal{A}} \widehat{\textbf{U}}_Q^+ \right)
                     = \text{int} \left( \bigcup_{Q \in \mathcal{A}} \bigcup_{I \in \mathcal{W}_Q^+} I^{***} \right)
\end{align*}
and $D_\mathcal{A}(Y) \coloneqq \dist(Y,\partial \Omega_\mathcal{A}^*)$. Recall the definitions of $I^{**}$ and $I^{***}$ from Section \ref{subsection:non-tangential}.
Let us fix a cube $P \in \mathcal{A}$ and a point $Y \in \widehat{U}_P^+ = \bigcup_{I \in \Wc_P^+} I^{**}$.
We now claim that $\delta(Y) \lesssim D_\mathcal{A}(Y)$ . 
We notice first that although the regions $\widehat{\textbf{U}}_Q^+$ may overlap, we have $\ell(Q) \approx \ell(Q') \approx \ell(P)$ for all
overlapping regions $\widehat{\textbf{U}}_Q^+$ and $\widehat{\textbf{U}}_{Q'}^+$ such that 
$Y \in \widehat{\textbf{U}}_Q^+ \cap \widehat{\textbf{U}}_{Q'}^+$ (see (3.2), (3.8) and related estimates in \cite{hofmannmartellmayboroda}). 
Also, the fattened Whitney boxes $I^{***}$ may overlap, but we have
$\ell(I^{***}) \approx \ell(I) \approx \ell(J) \approx \ell(J^{***}) \approx  \ell(P)$ if 
$Y \in I^{***} \cap J^{***}$. By a simple geometrical consideration we know that
$$
  \dist(Y,\partial I^{***}) \approx_{\tau}  \ell(I). 
$$
It now holds that $D_\mathcal{A}(Y) = \dist(Y,\partial I^{***})$ for some $I^{***} \ni Y$ or 
$D_\mathcal{A}(Y) \ge \dist(Y,\partial I^{***})$ for every such $I^{***}$. In particular, we have 
\begin{align*}
  D_\mathcal{A}(Y) &\ge \inf_{Q \in \mathcal{A}, Y \in \widehat{\textbf{U}}_Q^+} \inf_{I \in \mathcal{W}_Q^+} \dist(Y,\partial I^{***}) \\
                   &\approx \inf_{Q \in \mathcal{A}, Y \in \widehat{\textbf{U}}_Q^+} \inf_{I \in \mathcal{W}_Q^+} \ell(I)
                   \approx \inf_{Q \in \mathcal{A}, Y \in \widehat{\textbf{U}}_Q^+} \ell(Q)
                   \approx \ell(P).
\end{align*}
Now we can take any $I \in \Wc_P^+$ such that $Y \in I^{**}$ and notice that 
$\ell(P) \approx \ell(I) \approx \ell(I^{**}) \approx \text{dist}(I^{**},\partial \Omega) \approx \dist(Y, \partial \Omega)$.
Hence $D_\mathcal{A}(Y) \gtrsim \delta(Y)$ for every $Y \in \widehat{U}_P^+$.

\

\textbf{Part 3: The sum $I_{Q_0}$.} To simplify the notation, let us write
\begin{align*}
  \Rc_0^+ \coloneqq \{R \in \Rc^+ \cap \Sc_0 \colon R \subset Q_0, \pi_\Pc R = \pi_\Pc Q_0\}.
\end{align*}
We consider the region $\Omega^{***}$,
$$
\Omega^{***} := \text{int} \left( \bigcup_{R \in \Rc_0^+} \widehat{\textbf{U}}_R^+\right)
$$
and set $D(Y) \coloneqq \text{dist}(Y,\partial \Omega^{***})$ for every $Y \in \Omega$. Suppose that $R \in \Rc_0^+$. 
By Part 2, we know that
\begin{align}
  \label{distance_estimate} \delta(Y) \lesssim D(Y) \ \ \ \ \text{for every } Y \in \widehat{U}_R^+.
\end{align}
We also notice that
\begin{align*}
  \Omega^{***} = \text{int} \left(\bigcup_{R \in \Rc_0^+} \widehat{\textbf{U}}_R^+ \right)
               \subset \text{int} \left( \bigcup_{R \in \Rc_0^+} \bigcup_{x \in R} \Gamma(x) \right),
\end{align*}
so we have
\begin{align}
  \nonumber \sup_{X \in \Omega^{***}} |u(X)|
  = \sup_{R \in \Rc_0^+} \sup_{X \in \widehat{\textbf{U}}_R^+} |u(X)|
  &\le \sup_{R \in \Rc_0^+} \inf_{x \in R} N_*u(x) \\
  \label{control_in_omega***} &\le \sup_{R \in \Rc_0^+} M_\D(N_*u)(R)
  \lesssim M_\D(N_*u)(\pi_\Pc Q_0).
\end{align}
In the last inequality we used the definition of $\Rc_0^+$ (see Remark \ref{remark:simplification}).

By \cite[(5.8)]{hofmannmartellmayboroda} (or \cite[Section 4]{hofmannmartell}), we have
\begin{align}
  \label{estimate:oscillation_square} \left( \osc{U_R^+} u \right)^2 \lesssim \ell(R)^{-n} \iint_{\widehat{U}_R^+} |\nabla u(Y)|^2 \delta(Y) \, dY
\end{align}
for every $R \in \Rc^+$. Notice also that if $R \in \Rc_0^+$, then by the definition of the numbers $M_D(N_*u)(Q)$ we have $M_\D(N_*u)(\pi_\Pc Q_0) \le M_\D(N_*u)(R)$ simply because $R \subset \pi_\Pc Q_0$. Thus, using (A) the definition of the numbers $M_\D(N_*u)(Q)$, (B) the 
ADR property of $E$, (C) the definition of the collection $\Rc^+$ and (D) the bounded overlap of the regions $\widehat{U}_R^+$ we get
\begin{align}
  \label{inequality_chain1} M_\D(N_*u)(\pi_\Pc Q_0)^2 I_{Q_0}
  &\overset{\text{(A)}}{\le} \sum_{R \in \Rc_0^+} M_\D(N_*u)(R)^2 \sigma(R) \\
  \nonumber &\overset{\text{(B)}}{\lesssim} \sum_{R \in \Rc_0^+} M_\D(N_*u)(R)^2 \ell(R)^n \\
  \nonumber &\overset{\text{(C)}, \eqref{estimate:oscillation_square}}{\lesssim} \frac{1}{\varepsilon^2} \sum_{R \in \Rc_0^+} \iint_{\widehat{U}_R^+} |\nabla u(Y)|^2 \delta(Y) \, dY  \\
  \nonumber &\overset{\eqref{distance_estimate}}{\lesssim} \frac{1}{\varepsilon^2} \sum_{R \in \Rc_0^+} \iint_{\widehat{U}_R^+} |\nabla u(Y)|^2 D(Y) \, dY  \\
  \nonumber &\overset{\text{(D)}}{\lesssim} \frac{1}{\varepsilon^2} \iint_{\Omega^{***}} |\nabla u(Y)|^2 D(Y) \, dY
\end{align}
Since $Q_0 \in \Rc$, we notice that the collection $\Rc_0^+$ forms a 
semi-coherent subregime of $\Sc_0$. Thus, by \cite[Lemma 3.24]{hofmannmartellmayboroda}, the set $\Omega^{***}$ is a chord-arc domain (i.e. NTA domain with ADR 
boundary). Furthermore, by \cite[Theorem 1.2]{azzametal}, $\partial \Omega^{***}$ is UR. Since $\Omega^{***} \subset B(x_{Q_0}, C\ell(Q_0))$ 
for a suitable structural constant $C$ (see \cite[(3.14)]{hofmannmartellmayboroda}), the ADR property of $\partial \Omega$ and \cite[Theorem 1.1]{hofmannmartellmayboroda}
give us
\begin{align}
  \label{inequality_chain2} \frac{1}{\varepsilon^2} \iint_{\Omega^{***}} |\nabla u(Y)|^2 D(Y) \, dY 
  \lesssim \frac{1}{\varepsilon^2} \|u\|_{L^\infty(\Omega^{***})}^2 \cdot \sigma(Q_0) 
  \overset{\eqref{control_in_omega***}}{\lesssim} \frac{1}{\varepsilon^2} M_\D(N_*u)(\pi_\Pc Q_0)^2 \cdot \sigma(Q_0).
\end{align}
Since the numbers $M_\D(N_*u)(\pi_\Pc Q_0)^2$ cancel from \eqref{inequality_chain1} and \eqref{inequality_chain2}, this concludes the proof of the lemma.
\end{proof}

Since the bad collection $\Bc$ in the bilateral corona decomposition satisfies a Carleson packing condition, we immediately get the following corollary:

\begin{corollary}
  \label{corollary:bad_cubes_carleson}
  For every $Q_0 \in \D(E)$ we have
  \begin{align}
    \label{corollary:carleson_bad_oscillation} \sum_{R \in (\Rc \cup \Bc), R \subseteq Q_0} \sigma(R) \lesssim \frac{1}{\varepsilon^2} \sigma(Q_0).
  \end{align}
\end{corollary}

\section{Generation cubes}

For every stopping time regime $\Sc$, we construct a collection of \emph{generation cubes} $G(\Sc)$ as in \cite[Section 5]{hofmannmartellmayboroda} but with modified stopping conditions. For clarity, let us repeat the key details and definitions from \cite[Section 5]{hofmannmartellmayboroda} here. We set $Q^0 \coloneqq Q(\Sc)$ and $G_0 \coloneqq \{Q^0\}$, start subdividing $Q^0$ dyadically and stop when we reach a cube $Q \in \D_{Q^0}$ for which at least one of the following conditions holds:
\begin{enumerate}
  \item[(1)] $Q$ is not in $\Sc$,
  \item[(2)] $|u(Y_Q^+) - u(Y_{Q^0}^+)| > \varepsilon M_\D(N_*u)(Q)$,
  \item[(3)] $|u(Y_Q^-) - u(Y_{Q^0}^-)| > \varepsilon M_\D(N_*u)(Q)$.
\end{enumerate}
The points $Y_Q^{\pm}$ were defined in Section \ref{subsection:whitney_regions}. We denote the collection of maximal subcubes of $Q^0$ extracted by these stopping time conditions by $\Fc_1 = \Fc_1(Q^0)$ and we let $G_1 = G_1(Q^0) \coloneqq \Fc_1 \cap \Sc$ be the collection of \emph{first generation cubes}. We notice that the collection of subcubes of $Q^0$ that are not contained in any stopping cube $Q \in \Fc_1$ form a semicoherent subregime of $\Sc$. We denote this subregime by $\Sc' = \Sc'(Q^0)$.

If $G_1$ is non-empty, we repeat the construction above for the cubes $Q^1 \in G_1$ but replace $Y_{Q^0}^\pm$ by $Y_{Q^1}^\pm$ in conditions (2) and (3). Continuing like this gives us collections $G_k$ for $k \ge 0$ (notice that starting from some $k$ the collections might be empty), where
\begin{align*}
  G_{k+1}(Q^0) \coloneqq \bigcup_{Q^k \in G_k(Q^0)} G_1(Q^k).
\end{align*}
To emphasize the dependency on $\Sc$, we denote
\begin{align*}
  G_k(\Sc) \coloneqq G_k(Q(\Sc)),
\end{align*}
and we set the collection of all generation cubes to be
\begin{align*}
  G^* \coloneqq \bigcup_{\Sc} G(\Sc).
\end{align*}
By this construction, we have
\begin{align}
  \label{decomposition:stopping_time_regimes} \Sc = \bigcup_{Q \in G(\Sc)} \Sc'(Q)
\end{align}
for each stopping time regime $\Sc$, where $\Sc'(Q)$ is a semicoherent subregime of $\Sc$ with maximal element $Q$ and the subregimes $\Sc'(Q)$ are disjoint.

Our next goal is to prove that the collection $G^*$ satisfies a Carleson packing condition:
\begin{lemma}
  \label{lemma:stopping_carleson}
  For every $Q_0 \in \D$ we have
  \begin{align}
    \sum_{S \in G^*, S \subseteq Q_0} \sigma(S) \lesssim \frac{1}{\eps^2} \sigma(Q_0).
  \end{align}
\end{lemma}
Before the proof, let us make two observations that help us to simplify the proof. 
\begin{enumerate}
  \item[1)] By arguing as in the proof of Lemma \ref{lemma:oscillation_carleson}, we may assume that $Q_0 \in G^*$ and it suffices 
            to show that
            \begin{align*}
              \underset{\pi_\Pc S = \pi_\Pc Q_0}{\sum_{S \in G^* \cap \Sc_0, S \subset Q_0}} \sigma(S) \lesssim \frac{1}{\eps^2} \sigma(Q_0),
            \end{align*}
            where $\Sc_0$ is the unique stopping time regime such that $Q_0 \in \Sc_0$.
  \item[2)] For every $k \ge 0$ and $S \in G_k(\Sc_0)$, let $G_1(S) \subset G(\Sc_0)$ be the $G^*$-children of $S$, i.e. the cubes $P \in G_{k+1}(\Sc_0)$ such that $P \subsetneq S$. 
            For each such $S$ we have
            \begin{align}
              \label{stopping_carleson_estimate} M_\D(N_* u)(S)^2 \underset{\pi_\Pc Q = \pi_\Pc Q_0}{\sum_{Q \in G_1(S)}} \sigma(Q) \lesssim \frac{1}{\eps^2} \iint_{\Omega_{\mathscr{S}(S)}} |\nabla u(Y)|^2 \delta(Y) \, dY,
            \end{align}
            where $\mathscr{S}(S) \coloneqq \Sc'(S) \cap \{Q \in \D \colon \pi_\Pc Q = \pi_\Pc Q_0\}$ is a semicoherent subregime of $\Sc_0$ 
            and $\Omega_{\mathscr{S}(S)}$ is the associated sawtooth region (see \eqref{defin:sawtooth}). The estimate \eqref{stopping_carleson_estimate} is a counterpart of \cite[Lemma 5.11]{hofmannmartellmayboroda} and it follows easily from the original proof.
            To be a little more precise, instead of having $\eps^2 \le 100|u(Y_Q^+) - u(Y_S^+)|^2$ for every $Q \in G_1(S)$ as in 
            \cite[(5.13)]{hofmannmartellmayboroda}, we have $\eps^2 M_\D(N_*u)(S)^2 \le \eps^2 M_\D(N_*u)(Q)^2 \le |u(Y_Q^+) - u(Y_S^+)|^2$ for every $Q \in G_1(S)$.
            The rest of the proof works as it is.
\end{enumerate}

\begin{proof}[Proof of Lemma {\ref{lemma:stopping_carleson}}]
  Let us follow the arguments in the proof of \cite[Lemma 5.16]{hofmannmartellmayboroda} and write
  \begin{align*}
    \underset{\pi_\Pc S = \pi_\Pc Q_0}{\sum_{S \in G^* \cap \Sc_0, S \subset Q_0}} \sigma(S) 
    &= \sum_{k \ge 0} \underset{\pi_\Pc S = \pi_\Pc Q_0}{\sum_{S \in G_k(Q_0)}} \sigma(S) \\
    &= \sigma(Q_0) + \sum_{k \ge 1} \sum_{S' \in G_{k-1}(Q_0)} \underset{\pi_\Pc S = \pi_\Pc Q_0}{\sum_{S \in G_1(S')}} \sigma(S) \eqqcolon \sigma(Q_0) + I.
  \end{align*}
  Using \eqref{stopping_carleson_estimate} and the definition of the sawtooth regions gives us
  \begin{align}
    \nonumber M_\D(N_* u)(\pi_\Pc Q_0)^2 I
    &\overset{\eqref{stopping_carleson_estimate}}{\lesssim} \frac{1}{\eps^2} \sum_{k \ge 1} \sum_{S' \in G_{k-1}(Q_0)} \iint_{\Omega_{\mathscr{S}(S')}} |\nabla u(Y)|^2 \delta(Y) \, dY \\
    \label{estimate:triple_sum} &\le \frac{1}{\eps^2} \sum_{k \ge 1} \sum_{S' \in G_{k-1}(Q_0)} \underset{\pi_\Pc S = \pi_\Pc Q_0}{\sum_{S \in \Sc'(S')}} \iint_{U_S} |\nabla u(Y)|^2 \delta(Y) \, dY
  \end{align}
  We denote $\Omega_0 \coloneqq \bigcup_{S \in G^*_{Q_0}} U_S$ where $G^*_{Q_0} \coloneqq \{S \in \D \colon \pi_\Pc S = \pi_\Pc Q_0\} \cap \bigcup_{k \ge 1} \bigcup_{S' \in G_{k-1}(Q_0)} \Sc'(S')$. 
  By the construction, $\bigcup_{k \ge 1} \bigcup_{S' \in G_{k-1}(Q_0)} \Sc'(S')$ is a coherent subregime of $\Sc_0$ with maximal element $Q_0$ 
  and thus, $G^*_{Q_0}$ is a semicoherent subregime of $\Sc_0$. In particular, the sawtooth region $\Omega_0$ splits into 
  two chord-arc domains $\Omega_0^\pm$ by \cite[Lemma 3.24]{hofmannmartellmayboroda}. Furthermore, by \cite[Theorem 1.2]{azzametal}, both 
  $\partial \Omega_0^+$ and $\partial \Omega_0^-$ are UR. We also note that $\Omega_0 \subset B(x_{Q_0}, C \ell(Q_0))$ (see \cite[(3.14)]{hofmannmartellmayboroda}). 
  Thus, since the triple sum in \eqref{estimate:triple_sum} runs over a collection of disjoint cubes, we can use the 
  bounded overlap of the Whitney regions, \cite[Theorem 1.1]{hofmannmartellmayboroda} and the ADR property of $E$ to show that
  \begin{align*}
    \frac{1}{\eps^2} \sum_{k \ge 1} \sum_{S' \in G_{k-1}(Q_0)} \underset{\pi_\Pc S = \pi_\Pc Q_0}{\sum_{S \in \Sc'(S')}} \iint_{U_S} |\nabla u(Y)|^2 \delta(Y) \, dY
    &\lesssim \frac{1}{\eps^2} \iint_{\Omega_0} |\nabla u(Y)|^2 \delta(Y) \, dY \\
    &\lesssim \frac{1}{\eps^2} \|u\|_{L^\infty(\Omega_0)}^2 \sigma(Q_0).
  \end{align*}
  Since $\pi_\Pc S = \pi_\Pc Q_0$ for every $S \in G^*_{Q_0}$, by \eqref{stopping_condition_1} we have $M_\D(N_*u)(S) \le 2M_\D(N_*u)(\pi_\Pc Q_0)$ for every $S \in G^*_{Q_0}$.
  In particular:
  \begin{align*}
    \|u\|_{L^\infty(\Omega_0)}^2 &\le \sup_{S \in G^*_{Q_0}} \sup_{Y \in U_S} |u(Y)|^2 \\
                                 &\le \sup_{S \in G^*_{Q_0}} \inf_{x \in S} N_*u(x)^2 \\ 
                                 &\le \sup_{S \in G^*_{Q_0}} M_\D(N_*u)(S)^2 \lesssim M_\D(N_*u)(\pi_\Pc Q_0)^2.
  \end{align*}  
  Since the numbers $M_\D(N_* u)(\pi_\Pc Q_0)^2$ cancel out, we have proven the Carleson packing condition of $G^*$.
\end{proof}

\section{Construction of the approximating function}
\label{section:construction}

Before we construct the function, we prove the following technical lemma related to the modified cones $\Gamma_\alpha(x)$ that we 
defined in Section \ref{subsection:non-tangential}. Recall that
\begin{align}
  \label{defin:large_cones} \Gamma_\alpha(x) = \bigcup_{Q \in \D(E), Q \ni x} \bigcup_{\substack{P \in \D(E), \\ \ell(P) = \ell(Q), \\ \alpha \Delta_Q \cap P \neq \emptyset}} \widehat{\textbf{U}}_P.
\end{align}
  
\begin{lemma}
  \label{lemma:large_cones}
  There exists a uniform constant $\alpha_0 > 0$ such that the following holds: if $Q \in \D(E)$ is any cube and 
  $P \in G^*$ is a generation cube such that $\ell(Q) \le \ell(P)$ and $\Omega_{\Sc'(P)} \cap T_Q \neq \emptyset$, 
  then $X_P^\pm, Y_P^\pm \in \Gamma_{\alpha_0}(x)$ for every $x \in Q$.
\end{lemma}

\begin{proof}
  We start by noticing that there exists $\alpha > 0$ (depending only on the structural constants) such that
  \begin{align}
    \label{implication:union} \text{if } P \text{ appears in the union } \eqref{defin:large_cones}\text{, then also } \widetilde{P} \text{ appears in the same union},
  \end{align}
  where $\widetilde{P}$ is the dyadic parent of $P$. 
  Indeed, if we have $Q, P \in \D(E)$, $x \in Q$, $\ell(Q) = \ell(P)$ and $\alpha \Delta_Q \cap P \neq \emptyset$,
  then also $x \in \widetilde{Q}$, $\ell(\widetilde{Q}) = \ell(\widetilde{P})$ and $\alpha \Delta_{\widetilde{Q}} \cap \widetilde{P} \neq \emptyset$. The last
  claim follows from the fact that $\emptyset \neq \alpha \Delta_Q \cap P \subset \alpha \Delta_{\widetilde{Q}} \cap \widetilde{P}$.

  Let us then prove the claim of the lemma by following the argument in the proof of \cite[Lemma 5.20]{hofmannmartellmayboroda}. 
  Since $\Omega_{\Sc'(P)} \cap T_Q \neq \emptyset$, there exist cubes $P' \in \Sc'(P)$ and $Q' \subset Q$ such that 
  $U_{P'} \cap U_{Q'} \neq \emptyset$. By the properties of the Whitney regions, we have $\dist(Q',P') \lesssim \ell(Q') \approx \ell(P')$.
  Let us consider two cases:
  \begin{enumerate}
    \item[i)] Suppose that $\ell(P') \ge \ell(Q)$. Then there exists a cube $Q''$ such that $Q \subset Q''$ and $\ell(Q'') = \ell(P')$. 
              Since $Q' \subset Q''$, we have $\dist(Q'',P') \le \dist(Q',P') \lesssim \ell(Q') \le \ell(Q'')$. Thus, for a large enough $\alpha_0$, 
              we have $\widehat{\bf{U}}_{P'} \subset \Gamma_{\alpha_0}(x)$ for every $x \in Q$ and the claim follows from \eqref{implication:union}.
    
    \item[ii)] Suppose that $\ell(P') < \ell(Q)$. Then by the semicoherency of $\Sc'(P)$, there exists a cube $P'' \in \Sc'(P)$ 
               such that $P' \subset P'' \subset P$ and $\ell(P'') = \ell(Q)$. Since $P' \subset P''$ and $Q' \subset Q$, we know 
               that $\dist(P'',Q) \le \dist(P',Q') \lesssim \ell(Q') \le \ell(Q)$. Thus, for a large enough $\alpha_0$, we have 
               $\widehat{\bf{U}}_{P''} \subset \Gamma_{\alpha_0}(x)$ for every $x \in Q$. Again, the claim follows now from \eqref{implication:union}.
  \end{enumerate}
\end{proof}

\subsection{Constructing the function in $T_{Q_0}$}
\label{section:construction_local}

In this section we adopt the terminology from other papers (including \cite{hofmannmartellmayboroda}) and say that a component $U_{Q}^i$ is \emph{blue} if $\text{osc}_{U_Q^i} u \le \varepsilon M_\D(N_*u)(Q)$ 
and \emph{red} if $\text{osc}_{U_Q^i} u > \varepsilon M_\D(N_*u)(Q)$.

We recall the construction of the local functions $\varphi_0$, $\varphi_1$ and $\varphi$ from 
\cite[Section 5]{hofmannmartellmayboroda}. We start by defining an ordered family of good cubes $\{Q_k\}_{k \ge 1}$ relative 
to a fixed cube $Q_0 \in \D$. If $Q_0 \in \Gc$, then $Q_0 \in \Sc$ for some stopping time regime $\Sc$ and thus, 
$Q_0 \in \Sc_1'$ for some subregime in \eqref{decomposition:stopping_time_regimes}. In this case, we set $Q_1 = Q(\Sc_1')$.
If $Q_0 \notin \Gc$, then we let $Q_1$ be any good subcube of $Q_0$ such that $Q_1$ is maximal 
with respect to the side length; such a cube much exist since $\Bc$ is Carleson. Since $Q_1 \in \Gc$, we have 
$Q_1 \in \Sc$ for some stopping time regime $\Sc$, and by the coherency of $\Sc$, we have $Q_1 = Q(\Sc_1')$ 
for some subregime in \eqref{decomposition:stopping_time_regimes}. Once the cube $Q_1$ has been chosen in these two cases,
we let $Q_2$ be a subcube of maximum side length in $(\D_{Q_0} \cap \Gc) \setminus \Sc_1'$ and so on. This gives us a 
sequence of cubes $Q_k \in \Gc$ such that $\ell(Q_1) \ge \ell(Q_2) \ge \ell(Q_3) \ge \cdots$, $Q_k = Q(\Sc_k')$ and 
$\Gc \cap \D_{Q_0} \subset \bigcup_{k \ge 1} \Sc_k'$. We define recursively
\begin{align*}
  A_1 \coloneqq \Omega_{\Sc_1'}, \ \ \ \ \ A_k \coloneqq \Omega_{\Sc_k'} \setminus \left( \bigcup_{j=1}^{k-1} A_j \right), \ k \ge 2.
\end{align*}
and
\begin{align*}
  A_1^\pm \coloneqq \Omega_{\Sc_1'}^\pm, \ \ \ \ \ A_k^\pm \coloneqq \Omega_{\Sc_k'}^\pm \setminus \left( \bigcup_{j=1}^{k-1} A_j \right), \ k \ge 2,
\end{align*}
where
\begin{align*}
  \Omega_{\Sc_k'} \coloneqq \text{int}\left( \bigcup_{Q \in \Sc_k'} U_Q^\pm \right).
\end{align*}
We also set
\begin{align*}
  \Omega_0 \coloneqq \bigcup_k \Omega_{\Sc_k'} = \bigcup_k A_k \ \ \ \ \ \text{ and } \ \ \ \ \ \Omega_0^\pm \coloneqq \bigcup_k A_k^\pm.
\end{align*}
We now define $\varphi_0$ on $\Omega_0$ by setting
\begin{align*}
  \varphi_0 \coloneqq \sum_k \left(u(Y_{Q_k}^+)1_{A_k^+} + u(Y_{Q_k}^-)1_{A_k^-}\right).
\end{align*}
As for the rest of the subcubes of $\D_{Q_0}$, we let $\{Q(k)\}_k$ be some fixed enumeration of the cubes 
$(\Rc \cup \Bc) \cap \D_{Q_0}$ and define recursively
\begin{align*}
  V_1 \coloneqq U_{Q(1)}, \ \ \ \ \ V_k \coloneqq U_{Q(k)} \setminus \left( \bigcup_{j=1}^{k-1} V_j \right), \ k \ge 2.
\end{align*}
Each Whitney region $U_{Q(k)}$ splits into a uniformly bounded number of connected components $U_{Q(k)}^i$. Thus, we may 
further split
\begin{align*}
  V_1^i \coloneqq U_{Q(1)}^i, \ \ \ \ \ V_k^i \coloneqq U_{Q(k)}^i \setminus \left( \bigcup_{j=1}^{k-1} V_j \right), \ k \ge 2
\end{align*}
and then define 
\begin{align*}
  \varphi_1(Y) \coloneqq \left\{ \begin{array}{cl}
                                   u(Y), & \text{ if } U_{Q(k)}^i \text{ is red} \\
                                   u(X_I), & \text{ if } U_{Q(k)}^i \text{ is blue}
                                 \end{array} \right., Y \in V_k^i,
\end{align*}
on each $V_k^i$, where $X_I$ is the center of a fixed Whitney cube $I \subset U_{Q(k)}^i$. 
We then denote $\Omega_1 \coloneqq \text{int} \left( \bigcup_{Q \in \left( \mathcal{B} \cup \Rc \right) \cap \D_{Q_0}} U_Q \right) = \text{int} \left( \bigcup_k V_k \right)$,
set the values of $\varphi_0$ and $\varphi_1$ to be $0$ outside their original domains of definition
and define the function $\varphi$ on the Carleson box $T_{Q_0}$ as
\begin{align*}
  \varphi(Y) \coloneqq \left\{ \begin{array}{cl}
                                 \varphi_0(Y), & Y \in T_{Q_0} \setminus \overline{\Omega_1} \\
                                 \varphi_1(Y), & Y \in \Omega_1
                               \end{array} \right. ,
\end{align*}
From the point of view of $\Cc_\D$, the values of $\varphi$ on the boundary of $\Omega_1$ are not important since 
the $(n+1)$-dimensional measure of $\partial \Omega_1$ is $0$. Thus, we may simply set $\varphi|_{\partial \Omega_1} = u$ 
since this is convenient from the point of view of $N_*(u - \varphi)$.

\subsection{Verifying the estimates on $Q_0$}

Let us fix a cube $Q_0 \in \D(E)$. We start by verifying the following three estimates on $Q_0$.

\begin{lemma}
  \label{lemma:local_pointwise_bounds}
  Suppose that $x \in Q_0$, $Q' \in \D_{Q_0}$ and $\overrightarrow{\Psi} \in C_0^1(W_{Q'})$ with $\|\overrightarrow{\Psi}\|_{L^\infty} \le 1$, where $W_{Q'} \subset \Omega$ is any bounded and open set satisfying $T_{Q'} \subset W_{Q'}$. Then the following estimates hold:
  \begin{enumerate}
    \item[i)] $N_*(1_{T_{Q_0}} (u - \varphi))(x) \le \eps M_\D(N_* u)(x)$,
    \item[ii)] \begin{align*}
                 \iint_{T_{Q'} \setminus \overline{\Omega_1}} \varphi_0 \text{div} \overrightarrow{\Psi} \lesssim \frac{1}{\eps^2} \int_{\beta \Delta_{Q'}} N_*^{\alpha_0}u \, d\sigma,
               \end{align*}
    \item[iii)] \begin{align*}
                  \iint_{T_{Q'}} \varphi_1 \text{div} \overrightarrow{\Psi} \lesssim \frac{1}{\eps^2} \int_{\beta \Delta_{Q'}} N_* u \, d\sigma,
                \end{align*}
  \end{enumerate}
  where $\beta > 0$ is a uniform constant and $\alpha_0 > 0$ is the constant in Lemma \ref{lemma:large_cones}.
\end{lemma}

\begin{proof} \

  \begin{enumerate}
    \item[i)] Let us estimate the quantity $|u(Y) - \varphi(Y)|$ for different $Y \in T_{Q_0}$.
              \begin{enumerate}
                \item[$\bullet$] Suppose that $Y \in V_k^i$ such that $U_{Q(k)}^i$ is a red component. Then we have $\varphi(Y) = u(Y)$ and 
                                 $|u(Y) - \varphi(Y)| = 0$.
                
                \item[$\bullet$] Suppose that $Y \in V_k^i$ such that $U_{Q(k)}^i$ is a blue component. Then $\varphi(Y) = u(X_I)$ for a Whitney cube 
                                 $I \subseteq U_{Q(k)}^i$ and $|u(Y) - \varphi(Y)| \le \text{osc}_{U_{Q(k)}^i} u \le \eps M_\D(N_*u)(Q(k))$.
                                  
                \item[$\bullet$] Suppose that $Y \in T_{Q_0} \setminus \overline{\Omega_1}$. Then $Y \in A_k^\pm$ for some $k$ such that $Q_k \notin \Rc$.
                                 Without loss of generality, we may assume that $Y \in A_k^+$. Now $\varphi(Y) = u(Y_{Q_k}^+)$ and, since $Q_k \notin \Rc$,
                                 we have $|u(Y) - \varphi(Y)| \le \text{osc}_{U_{Q_k}^+} \le \eps M_\D(N_*u)(Q_k)$.
              \end{enumerate}
              Combining the previous estimates gives us
              \begin{align*}
                N_*(1_{T_{Q_0}}(u-\varphi))(x) &= \sup_{Y \in \Gamma(x) \cap T_{Q_0}} |u(Y) - \varphi(Y)| \\
                                               &= \sup_{\substack{Q \in \D_{Q_0} \\ Q \ni x}} \sup_{Y \in U_Q} |u(Y) - \varphi(Y)| \\
                                               &\le \sup_{\substack{Q \in \D_{Q_0} \\ Q \ni x}} \eps M_\D(N_*u)(Q) \\
                                               &\le \eps M_\D(N_*u)(x). 
              \end{align*}
              
    \item[ii)] We first notice that since $\Psi$ is compactly supported in $\Omega$, we have $
               \dist(\text{supp} \, \Psi, E) > 0$. Thus, for each $A_k$, the set $(T_{Q'} \cap A_k \cap \text{supp} \, \Psi) \setminus \overline{\Omega_1}$ consists of a union of boundedly overlapping 
               sets that are ``nice'' enough for integration by parts. The divergence theorem gives us
               \begin{align*}
                 \iint_{T_{Q'} \setminus \overline{\Omega_1}} \varphi_0 \, \text{div} \overrightarrow{\Psi}
                 &\le \sum_k \iint_{(T_{Q'} \cap A_k) \setminus \overline{\Omega_1}} \varphi_0 \, \text{div} \overrightarrow{\Psi} \\
                 &= \sum_k \iint_{(T_{Q'} \cap A_k) \setminus \overline{\Omega_1}}  \text{div}(\varphi_0 \overrightarrow{\Psi}) \\
                 &\le \sum_k \left( \iint_{\partial( (T_{Q'} \cap A_k^+) \setminus \overline{\Omega_1}))}  \varphi_0 \overrightarrow{\Psi} \cdot \overrightarrow{N}
                 + \iint_{\partial( (T_{Q'} \cap A_k^-) \setminus \overline{\Omega_1}))} \varphi_0 \overrightarrow{\Psi} \cdot \overrightarrow{N} \right) \\
                 &\le \sum_k |u(Y_{Q_k}^+)| \cdot \Hc^n( T_{Q'} \cap \partial (A_k^+ \setminus \overline{\Omega_1}) ) \\ 
                 &\quad + \sum_k |u(Y_{Q_k}^-)| \cdot \Hc^n( T_{Q'} \cap \partial (A_k^- \setminus \overline{\Omega_1}) ) \\
                 &\eqqcolon I^+ + I^-.
               \end{align*}
               We only consider the sum $I^+$ since the sum $I^-$ can be handled the same way as $I^+$. We get
               \begin{align*}
                 \Hc^n( T_{Q'} \cap \partial (A_k^+ \setminus \overline{\Omega_1}))
                 \le \Hc^n( T_{Q'} \cap \partial A_k^+ ) + \Hc^n( T_{Q'} \cap A_k^+ \cap \partial \Omega_1 )
               \end{align*}
               and thus, we have
               \begin{align*}
                 I^+ \le \sum_k |u(Y_{Q_k}^+)| \cdot \Hc^n( T_{Q'} \cap \partial A_k^+ )
                         + \sum_k |u(Y_{Q_k}^+)| \cdot \Hc^n( T_{Q'} \cap A_k^+ \cap \partial \Omega_1 )
                     \eqqcolon I_1^+ + I_2^+.
               \end{align*}
               Let us consider the sum $I_1^+$ first. We split
               \begin{align*}
                 I_1^+ = \sum_{k \colon Q_k \subset Q'} |u(Y_{Q_k}^+)| \cdot \Hc^n(T_{Q'} \cap \partial A_k^+)
                       + \sum_{k\colon Q_k \not\subset Q'} |u(Y_{Q_k}^+)| \cdot \Hc^n(T_{Q'} \cap \partial A_k^+)
                       \eqqcolon J_1^+ + J_2^+.
               \end{align*}
               By \cite[Proposition A.2, (5.21)]{hofmannmartellmayboroda} we know that $\partial A_k^+$ satisfies 
               an upper ADR bound. Thus, since $\partial(T_{Q'} \cap A_k^+) \subset \overline{\Omega_{\Sc_k'}}$ and $\diam(\Omega_{\Sc_k'}) \lesssim \ell(Q_k)$, we get
               \begin{align*}
                 J_1^+ \lesssim \sum_{k: Q_k \subset Q'} |u(Y_{Q_k}^+)| \cdot \ell(Q_k)^n
                       \approx \sum_{k: Q_k \subset Q'} |u(Y_{Q_k}^+)| \cdot \sigma(Q_k) 
                       \le \sum_{k: Q_k \subset Q'} \inf_{Q_k} N_*u \cdot \sigma(Q_k).
               \end{align*}
               Since the collection of generation cubes is $C\eps^{-2}$-Carleson by Lemma \ref{lemma:stopping_carleson}, it is $C\eps^2$-sparse by Theorem \ref{thm:sparse_carleson}. Thus, we get
               \begin{align*}
                 \sum_{k: Q_k \subset Q'} \inf_{Q_k} N_*u \cdot \sigma(Q_k)
                       &\lesssim \frac{1}{\eps^2} \sum_{k: Q_k \subset Q'} \inf_{Q_k} N_*u \cdot \sigma(E_{Q_k}) \\
                       &\le \frac{1}{\eps^2} \sum_{k: Q_k \subset Q'} \int_{E_{Q_k}} N_*u \, d\sigma \\
                       &\le \frac{1}{\eps^2} \int_{Q'} N_* u \, d\sigma
               \end{align*}
               Let us then consider the sum $J_2^+$. By the same argument as in \cite[p. 2370]{hofmannmartellmayboroda}, we know that 
               the number of the cubes $Q_k$ such that $T_{Q'} \cap \partial A_k^+ \neq \emptyset$ and $\ell(Q_k) \ge \ell(Q')$ is uniformly 
               bounded. Thus, by Lemma \ref{lemma:large_cones} and the fact that $\partial A_k^+$ satisfies an upper ADR bound (as we noted above), we get
               \begin{align*}
                 \sum_{\substack{k\colon Q_k \not\subset Q', \\ T_{Q'} \cap \partial A_k^+ \neq \emptyset, \\ \ell(Q') \le \ell(Q_k)}} |u(Y_{Q_k}^+)| \cdot \Hc^n(T_{Q'} \cap \partial A_k^+) 
                        &\le \sum_{\substack{k\colon Q_k \not\subset Q', \\ T_{Q'} \cap \partial A_k^+ \neq \emptyset, \\ \ell(Q') \le \ell(Q_k)}} \inf_{Q'} N_*^{\alpha_0} u \cdot \Hc^n(T_{Q'} \cap \partial A_k^+) \\
                       &\quad \lesssim \inf_{Q'} N_*^{\alpha_0} u \cdot \left( \diam(T_{Q'}) \right)^n \\
                       &\quad \approx \inf_{Q'} N_*^{\alpha_0} u \cdot \sigma(Q') \\
                       &\quad \le \int_{Q'} N_*^{\alpha_0}u \, d\sigma.
               \end{align*}
               For the cubes $Q_k$ in $J_2^+$ such that $\ell(Q_k) \le \ell(Q')$ we may use the same argument as in \cite[p. 2370]{hofmannmartellmayboroda} to see that
               every such cube is contained in some nearby cube $Q''$ of $Q'$ of the same side length as $Q'$ with $\dist(Q',Q'') \lesssim \ell(Q')$.
               The number of such $Q''$ is uniformly bounded. By using the same techniques as with the sum $J_1^+$, we get
               \begin{align*}
                 \sum_{\substack{k\colon Q_k \not\subset Q', \\ T_{Q'} \cap \partial A_k^+ \neq \emptyset, \\ \ell(Q') \ge \ell(Q_k)}} |u(Y_{Q_k}^+)| \cdot \Hc^n(T_{Q'} \cap \partial A_k^+)
                 &\lesssim \sum_{Q''} \frac{1}{\eps^2} \int_{Q''} N_*u \, d\sigma \\
                 &\le \frac{1}{\eps^2} \int_{\beta_0 \Delta_{Q'}} N_*u \, d\sigma
               \end{align*}
               for some uniform constant $\beta_0$. Thus, we get
               \begin{align*}
                 J_2^+ \lesssim \frac{1}{\eps^2} \int_{\beta_0 \Delta_{Q'}} N_*^{\alpha_0}u \, d\sigma.
               \end{align*}
               Let us then consider the sum $I_2^+$. We first notice that
               \begin{align*}
                 \Hc^n( T_{Q'} \cap A_k^+ \cap \partial \Omega_1 ) \le \sum_m \Hc^n( T_{Q'} \cap A_k^+ \cap \partial V_m ).
               \end{align*}
               Thus, we get
               \begin{align*}
                 I_2^+ &\le \sum_k \sum_m |u(Y_{Q_k}^+)| \cdot \Hc^n( T_{Q'} \cap A_k^+ \cap \partial V_m ) \\
                       &= \sum_{k: Q_k \subset Q'} \sum_m |u(Y_{Q_k}^+)| \cdot \Hc^n( T_{Q'} \cap A_k^+ \cap \partial V_m ) \\
                       &\quad + \sum_{k: Q_k \not\subset Q'} \sum_m |u(Y_{Q_k}^+)| \cdot \Hc^n( T_{Q'} \cap A_k^+ \cap \partial V_m ) \\
                       &\eqqcolon J_3^+ + J_4^+.
               \end{align*}
               Suppose that $A_k^+ \cap \partial V_m \neq \emptyset$. Then, by the construction, we have $\ell(Q(m)) \lesssim \ell(Q_k)$
               and $\dist(Q(m),Q_k) \lesssim \ell(Q_k)$. Thus, there exists a uniform constant $\beta_1 > 0$ such that 
               $Q(m) \subset \beta_1 \Delta_{Q_k}$ and the set $\beta_1 \Delta_{Q_k}$ can be covered by a uniformly bounded 
               number of disjoint cubes with approximately the same side length as $Q_k$.
               In particular, since $T_{Q'} \cap A_k^+ \cap \partial V_m$ satisfies an upper ADR bound for every $m$ by the construction 
               and \cite[(5.25), Proposition A.2]{hofmannmartellmayboroda}, we get
               \begin{align*}
                 J_3^+ &= \sum_{k: Q_k \subset Q'} \sum_m |u(Y_{Q_k}^+)| \cdot \Hc^n( T_{Q'} \cap A_k^+ \cap \partial V_m ) \\
                       &\lesssim \sum_{k: Q_k \subset Q'} |u(Y_{Q_k}^+)| \sum_{m: Q(m) \subset \beta_1 \Delta_{Q_k}} \ell(Q(m))^n \\
                       &\lesssim \sum_{k: Q_k \subset Q'} |u(Y_{Q_k}^+)| \sum_{m: Q(m) \subset \beta_1 \Delta_{Q_k}} \sigma(Q(m)) \\
                       &\overset{\eqref{corollary:carleson_bad_oscillation}}{\lesssim} \frac{1}{\eps^2} \sum_{k: Q_k \subset Q'} |u(Y_{Q_k}^+)| \cdot \sigma(Q_k).
               \end{align*}
               Now we can use exactly the same arguments as with the sum $J_1^+$ to see that
               \begin{align*}
                 J_3^+ \lesssim \frac{1}{\eps^2} \int_{Q'} N_* u \, d\sigma.
               \end{align*}
               Finally, let us handle the sum $J_4^+$. Just as above with the sum $J_3^+$, for some uniform constant $\beta_2 > 0$ we get
               \begin{align*}
                 \sum_{\substack{k: Q_k \not\subset Q' \\ \ell(Q') \le \ell(Q_k)}} \sum_m |u(Y_{Q_k}^+)| \cdot \Hc^n( T_{Q'} \cap A_k^+ \cap \partial V_m ) 
                 &\le \sum_{\substack{k: Q_k \not\subset Q' \\ T_{Q'} \cap A_k^+ \neq \emptyset \\ \ell(Q') \le \ell(Q_k)}} |u(Y_{Q_k}^+)| \sum_{m: V_m \subset \beta_2 \Delta_{Q'}} \sigma(Q(m)) \\
                 &\overset{\eqref{corollary:carleson_bad_oscillation}}{\lesssim} \frac{1}{\eps^2} \sum_{\substack{k: Q_k \not\subset Q' \\ T_{Q'} \cap A_k^+ \neq \emptyset \\ \ell(Q') \le \ell(Q_k)}} |u(Y_{Q_k}^+)| \cdot \sigma(Q') \\
                 &\overset{\ref{lemma:large_cones}}{\le} \frac{1}{\eps^2} \sum_{\substack{k: Q_k \not\subset Q' \\ T_{Q'} \cap A_k^+ \neq \emptyset \\ \ell(Q') \le \ell(Q_k)}} \inf_{Q'} N_*^{\alpha_0} u \cdot \sigma(Q') \\
                 &\lesssim \frac{1}{\eps^2} \int_{Q'} N_*^{\alpha_0} u \, d\sigma,
               \end{align*}
               where we used the fact that there exists only a uniformly bounded number of cubes $Q_k$ that satisfy the condition of the sum by \cite[Lemma 5.20]{hofmannmartellmayboroda}.
               By using the same argument as with the latter half of the sum $J_2^+$, we get the bound
               \begin{align*}
                 \sum_{\substack{k: Q_k \not\subset Q' \\ \ell(Q') \ge \ell(Q_k)}} \sum_m |u(Y_{Q_k}^+)| \cdot \Hc^n( T_{Q'} \cap A_k^+ \cap \partial V_m )
                 \lesssim \frac{1}{\eps^2} \int_{\beta_3 \Delta_{Q'}} N_*u \, d\sigma
               \end{align*}
               for some uniform contant $\beta_3 > 0$. Thus, we have
               \begin{align*}
                 J_4^+ \lesssim \frac{1}{\eps^2} \int_{\beta_3 \Delta_{Q'}} N_*^{\alpha_0}u \, d\sigma.
               \end{align*}
               Combining the estimates for $J_1^+$, $J_2^+$, $J_3^+$ and $J_4^+$ gives us the claim.

    \item[iii)] By \cite[(5.25)]{hofmannmartellmayboroda}, we have
                \begin{align}
                  \label{estimate:component_boundary} \Hc^n(\partial V_k^i) \le \Hc^n(\partial V_k) \lesssim \ell(Q(k))^n \approx \sigma(Q(k))
                \end{align}
                for every $Q(k)$ and $i$. We also note that $\partial T_{Q'}$ satisfies an upper ADR bound \cite[Proposition A.2]{hofmannmartellmayboroda}. 
                Recall that the function $\varphi_1$ is supported on $\Omega_1$. Thus, since the sets $V_l$ are disjoint, we get
                \begin{align*}
                  \iint_{T_{Q'}} \varphi_1 \text{div} \overrightarrow{\Psi}
                  &= \sum_l \iint_{T_{Q'} \cap V_l} \varphi_1 \text{div} \overrightarrow{\Psi} \\
                  &= \sum_l \sum_i \iint_{T_{Q'} \cap V_l^i} \varphi_1 \text{div} \overrightarrow{\Psi} \\
                  &= \sum_l \sum_i \left( \iint_{T_{Q'} \cap V_l^i} \text{div}( \varphi_1 \overrightarrow{\Psi} ) - \iint_{T_{Q'} \cap V_l^i} \nabla \varphi_1 \cdot \overrightarrow{\Psi} \right) \\
                  &\le \sum_l \sum_i \left( \left| \iint_{T_{Q'} \cap V_l^i} \text{div}( \varphi_1 \overrightarrow{\Psi} )\right| + \iint_{T_{Q'} \cap V_l^i} |\nabla \varphi_1| \right).
                \end{align*}
                Let us first assume that $U_{Q(l)}^i$ is a blue component. Recall that since the collection $\Rc \cup \Bc$ is $C\eps^{-2}$-Carleson by 
                Corollary \ref{corollary:bad_cubes_carleson}, it is $C\eps^2$-sparse by Theorem \ref{thm:sparse_carleson}. Thus, by the definition of 
                $\varphi_1$ and the divergence theorem, we have
                \begin{align*}
                  \left| \iint_{T_{Q'} \cap V_l^i} \text{div}( \varphi_1 \overrightarrow{\Psi} )\right| + \iint_{T_{Q'} \cap V_l^i} |\nabla \varphi_1|
                  &= \left| \iint_{T_{Q'} \cap V_l^i} \text{div}( \varphi_1 \overrightarrow{\Psi}) \right| \\
                  &\le \iint_{T_{Q'} \cap \partial V_l^i} |u(X_{I(l,i)})| \\
                  &\overset{\eqref{estimate:component_boundary}}{\le} \inf_{Q(l)} N_*u \cdot \sigma(Q(l)) \\
                  &\lesssim \frac{1}{\eps^2} \inf_{Q(l)} N_*u \cdot \sigma(E_{Q(l)}).
                \end{align*}
                Suppose then that $U_{Q(l)}^i$ is a red component. Since $\partial V_l^i \subset \Gamma(y)$ for every $y \in Q(l)$, 
                we get $| \iint_{T_{Q'} \cap V_l^i} \text{div}( u \overrightarrow{\Psi} )| \le \tfrac{1}{\eps^2} \inf_{Q(l)} N_*u \cdot \sigma(E_{Q(l)})$ by the same 
                argument as above. Also, by the definition of the function $\varphi_1$, Caccioppoli's inequality and the sparseness arguments, we have
                \begin{align*}
                  \iint_{T_{Q'} \cap V_l^i} |\nabla \varphi_1| 
                  &= \iint_{V_l^i} |\nabla u| \\
                  &\lesssim \left( \iint_{V_l^i} |\nabla u|^2 \right)^{1/2} \ell(Q(l))^{(n+1)/2} \\
                  &\lesssim \frac{1}{\ell(Q(l))} \left( \iint_{\widehat{U}_{Q(l)}} |u|^2 \right)^{1/2} \ell(Q(l))^{(n+1)/2} \\
                  &\lesssim \frac{1}{\ell(Q(l))} \left( \iint_{\widehat{U}_{Q(l)}} \inf_{Q(l)} (N_*u)^2 \right)^{1/2} \ell(Q(l))^{(n+1)/2} \\
                  &\lesssim \frac{1}{\ell(Q(l))} \inf_{Q(l)} (N_*u) \cdot \ell(Q(l))^{n+1} \\
                  &\approx \sigma(Q(l)) \cdot \inf_{Q(l)} (N_*u)
                  \lesssim \frac{1}{\eps^2} \sigma(E_{Q(l)}) \cdot \inf_{Q(l)} N_*u.
                \end{align*}
                Thus, since every Whitney region $U_Q$ has only a uniformly bounded number of components $U_Q^i$, we get
                \begin{align*}
                  \iint_{T_{Q'}} |\nabla \varphi_1| \lesssim \sum_l \frac{1}{\eps^2} \sigma(E_{Q(l)}) \cdot \inf_{Q(l)} N_*u.
                \end{align*}
                Since $V_l$ meets $T_{Q'}$, we know that $\dist(Q(l),Q') \lesssim \ell(Q')$. In particular, all the relevant cubes $Q(l)$ are contained 
                in some nearby cubes $Q''$ such that $\ell(Q'') \approx \ell(Q')$ and $\dist(Q'',Q') \lesssim \ell(Q')$. The number of such 
                $Q''$ is uniformly bounded. Thus, since the sets $E_{Q(l)}$ are disjoint, we get
                \begin{align*}
                  \sum_l \frac{1}{\eps^2} \sigma(E_{Q(l)}) \cdot \inf_{Q(l)} N_*u
                  \le \frac{1}{\eps^2} \sum_{Q''} \int_{Q''} N_* u \lesssim \frac{1}{\eps^2} \int_{\beta_0 \Delta_{Q'}} N_*u
                \end{align*}
                for some uniform constant $\beta_0$.
  \end{enumerate}
\end{proof}
Let us then consider the dyadic total variation of the whole approximating function $\varphi$:

\begin{proposition}
  \label{proposition:local_gradient_bound}
  Suppose that $Q' \in \D_{Q_0}$ and $\overrightarrow{\Psi} \in C_0^1(W_{Q'})$ with $\|\overrightarrow{\Psi}\|_{L^\infty} \le 1$, where $W_{Q'} \subset \Omega$ is any bounded and open set satisfying $T_{Q'} \subset W_{Q'}$. Then
  \begin{align*}
    \iint_{T_{Q'}} \varphi \, \text{div} \overrightarrow{\Psi} \lesssim \frac{1}{\eps^2} \int_{\beta \Delta_{Q'}} N_*^{\alpha_0} u \, d\sigma,
  \end{align*}
  where $\beta > 0$ is a uniformly bounded constant and $\alpha_0 > 0$ is the constant in Lemma \ref{lemma:large_cones}.
\end{proposition}

\begin{proof}
  We start by splitting the integral with respect to $\varphi_0$ and $\varphi_1$.
  \begin{align*}
    \iint_{T_{Q'}} \varphi \, \text{div} \overrightarrow{\Psi}
    = \iint_{T_{Q'} \setminus \overline{\Omega_1}} \varphi_0 \, \text{div} \overrightarrow{\Psi}
    + \iint_{T_{Q'} \cap \overline{\Omega_1}} \varphi_1 \, \text{div} \overrightarrow{\Psi}.
  \end{align*}
  For the first integral, we can simply use the part ii) of Lemma \ref{lemma:local_pointwise_bounds}. For the second integral we get 
  \begin{align*}
    \iint_{T_{Q'} \cap \overline{\Omega_1}} \varphi_1 \text{div} \overrightarrow{\Psi}
    &= \sum_k \iint_{V_k \cap T_{Q'}} \varphi_1 \text{div} \overrightarrow{\Psi} \\
    &= \sum_k \left( \iint_{V_k \cap T_{Q'}} \text{div}( \varphi_1 \overrightarrow{\Psi} ) - \iint_{V_k \cap T_{Q'}} \nabla \varphi_1 \cdot \overrightarrow{\Psi} \right) \\
    &\le \sum_k \left| \iint_{V_k \cap T_{Q'}} \text{div}( \varphi_1 \overrightarrow{\Psi} ) \right| + \sum_k \iint_{V_k \cap T_{Q'}} |\nabla \varphi_1|.
  \end{align*}
  The second sum is just as in the proof of part iii) of Lemma \ref{lemma:local_pointwise_bounds} and thus, we can bound it by
  $C\eps^{-2} \int_{\beta_0 \Delta_{Q'}} N_*u$. For the first sum, we use the divergence theorem and Theorem \ref{thm:sparse_carleson} and get
  \begin{align*}
    \sum_k \left| \iint_{V_k \cap T_{Q'}} \text{div}( \varphi_1 \overrightarrow{\Psi} ) \right|
    &\le \sum_k \iint_{\partial(V_k \cap T_{Q'})} \left| \varphi_1 \overrightarrow{\Psi} \cdot \overrightarrow{N} \right| \\
    &\le \sum_k \sup_{U_{Q(k)}} |u| \cdot \Hc^n(V_k \cap \partial T_{Q'}) \\
    &\le \sum_{k: \, \dist(Q(k),Q') \lesssim \ell(Q')} \inf_{Q(k)} N_* u \cdot \sigma(Q(k)) \\
    &\lesssim \frac{1}{\eps^2}\sum_{k: \, \dist(Q(k),Q') \lesssim \ell(Q')} \inf_{Q(k)} N_* u \cdot \sigma(E_{Q(k)}).
  \end{align*}
  By the structure of the Whitney regions, we know $V_k \cap T_{Q'} = \emptyset$ if $\ell(Q(k)) \gg \ell(Q')$ or $\dist(Q(k),Q') \gg \ell(Q')$. Thus, there exists a uniform constant $\beta_1 > 0$ such that $Q(k) \subset \beta_1 \Delta_{Q'}$ for every $k$ in the sum above. 
  We may cover $\beta_1 \Delta_{Q'}$ by a uniformly bounded number of disjoint cubes $P_j$ such that $\ell(P_j) \approx \ell(Q')$. 
  This gives us
  \begin{align*}
    \sum_{k: \, \dist(Q(k),Q') \lesssim \ell(Q')} \inf_{Q(k)} N_* u \cdot \sigma(E_{Q(k)}) &\le \sum_{k: \, \dist(Q(k),Q') \lesssim \ell(Q')} \int_{E_{Q(k)}} N_* u  \\
                                                            &\le \sum_j \int_{P_j} N_* u \, d\sigma \\
                                                            &\le \int_{\beta_2 \Delta_{Q'}} N_* u \, d\sigma
  \end{align*}
  for some uniform constant $\beta_2 \ge \beta_1$. Combining the previous bounds finishes the proof.
\end{proof}

\begin{remark}
  \label{remark:modified_local_bounds}
  We notice that the previous proposition holds also in the following form: If we have cubes $Q',Q_1,Q_2 \in \D_{Q_0}$ and $\overrightarrow{\Psi} \in C_0^1(W_{Q'})$ with $\|\overrightarrow{\Psi}\|_{L^\infty} \le 1$ for an open and bounded set $W_{Q'}$ containing $T_{Q'}$, then
  \begin{align*}
    \iint_{(T_{Q'} \cap T_{Q_1}) \setminus T_{Q_2}} \varphi \, \text{div} \overrightarrow{\Psi} \lesssim \frac{1}{\eps^2} \min\left\{ \int_{\beta_2 \Delta_{Q'}} N_* u \, d\sigma, \int_{\beta_2 \Delta_{Q_1}} N_* u \, d\sigma \right\}
  \end{align*}
  for some uniform constant $\beta_2$. Indeed, in the previous two proofs, we needed only the upper ADR estimates for 
  the boundaries of $A_m$ and $V_k$ and these estimates remain valid if we remove a finite number of 
  pieces whose boundaries satisfy an upper ADR estimate. By \cite[Proposition A.2]{hofmannmartellmayboroda}, 
  $\partial T_Q$ is ADR for every $Q \in \D(E)$. Also, by the stucture of the regions, these modified sets are ``nice'' enough to justify 
  integration by parts that we used in the proofs.
\end{remark}

\subsection{From local to global}

Let us now construct the global approximating function. Although our construction is a little different than the 
construction in \cite[p. 2373]{hofmannmartellmayboroda}, the basic ideas are the same.

\subsubsection{$E$ is a bounded set}
\label{subsection:function_construction_bounded}

Let us first assume that $\diam(E) < \infty$. In this case, we have a cube $Q_0 \in \D(E)$ such that $E = Q_0$ and 
$\ell(Q_0) \approx \diam(E)$. We now set
\begin{align*}
  \varphi(X) \coloneqq \left\{ \begin{array}{cl}
                                 \varphi_{Q_0}(X), &\text{ if } X \in T_{Q_0} \\
                                 u(X), &\text{ if } X \in \Omega \setminus T_{Q_0}
                               \end{array} \right. ,
\end{align*}
where $\varphi_{Q_0}$ is the function constructed in Section \ref{section:construction_local}.
By part i) of Lemma \ref{lemma:local_pointwise_bounds}, we have $N_*(u-\varphi)(x) \le \eps M_\D(N_*u)(x)$ on $E$.
As for the $\Cc_\D$ bound, we first notice that for any $Q \in \D_{Q_0}$ Proposition \ref{proposition:local_gradient_bound} gives us
\begin{align}
  \frac{1}{\sigma(Q)} \iint_{T_Q} |\nabla \varphi| \lesssim \frac{1}{\eps^2} M(N_*^{\alpha_0}u)(x)
\end{align}
for every $x \in Q$ since $\sigma(Q) \approx \sigma(\beta \Delta_Q)$. Let us now fix a cube $F_k \in \D^*$ (recall the definition of $\D^*$ in Section \ref{section:cc}), take any 
$\overrightarrow{\Psi} \in C_0^1(T_{F_k})$ with $\|\overrightarrow{\Psi}\|_{L^\infty} \le 1$
and modify the argument in \cite[p. 2353]{hofmannmartellmayboroda}.
We denote $R \coloneqq 2^k \diam(E)$ and thus have $T_{F_k} = B(z_0,R)$. By a suitable choice of parameters in the construction 
of the Whitney regions in \cite{hofmannmartellmayboroda}, the Carleson box $T_{Q_0}$ is so large that we may fix a ball 
$B(z_0,r) \subset T_{Q_0}$ such that $r \ge 2\diam(E)$. Because of this, we may fix a uniform constant $\alpha_1$ such that 
a small enlargement of $B(z_0,R) \setminus B(z_0,r)$ is contained in $\widehat{\Gamma}_{\alpha_1}(x)$ (recall the definition of $\widehat{\Gamma}_{\alpha_1}(x)$
in Section \ref{subsection:non-tangential}) for every $x \in E$. We split
\begin{align*}
  \frac{1}{\ell(F_k)^n} \iint_{T_{F_k}} \varphi \, \text{div} \overrightarrow{\Psi} = \frac{1}{\ell(F_k)^n} \iint_{T_{Q_0}} \varphi \, \text{div} \overrightarrow{\Psi} + \frac{1}{\ell(F_k)^n} \iint_{T_{F_k} \setminus T_{Q_0}} \varphi \, \text{div} \overrightarrow{\Psi}.
\end{align*}
By Proposition \ref{proposition:local_gradient_bound}, we can bound the first integral by 
$M(N_*^{\alpha_0}u)(x)$ for any $x \in Q_0$. As for the second integral, we use the smoothness of $u$, H\"older's inequality and Caccioppoli's inequality to get
\begin{align*}
  \iint_{T_{F_k} \setminus T_{Q_0}} \varphi \, \text{div} \overrightarrow{\Psi}
  &= \iint_{T_{F_k} \setminus T_{Q_0}} u \, \text{div} \overrightarrow{\Psi} \\
  &\le \iint_{B(z_0,R) \setminus T_{Q_0}} |\nabla u| \\
  &\le \iint_{B(z_0,R) \setminus B(z_0,r)} |\nabla u| \\
  &\lesssim \left( \iint_{B(z_0,R) \setminus B(z_0,r)} |\nabla u|^2 \right)^{1/2} R^{\tfrac{n+1}{2}} \\
  &\le \left( \sum_{0 \le j \le \log_2(R/r)} \iint_{2^j r \le |z_0-X| < 2^{j+1}r} |\nabla u(X)|^2 \right)^{1/2} R^{\tfrac{n+1}{2}} \\
  &\lesssim \inf_E N_*^{\alpha_1}u \cdot \left( \sum_{0 \le j \le \log_2(R/r)} (2^j r)^{n-1} \right)^{1/2} R^{\tfrac{n+1}{2}} \\
  &\lesssim \inf_E N_*^{\alpha_1}u \cdot R^{\tfrac{n-1}{2}} R^{\tfrac{n+1}{2}} \\
  &\le R^n M(N_*^{\alpha_1}u)(x)
\end{align*}
for every $x \in Q_0$. Combining the calculations and the cases gives us the desired $\Cc_\D$ bound.

\subsubsection{$E$ is an unbounded set}
\label{section:E_unbounded}

Suppose then that $\diam(E) = \infty$. We fix a sequence of cubes $Q_i \in \D(E)$, $i \in \N$, such that $\bigcup_i Q_i = E$ and
$Q_i \subsetneq Q_{i+1}$ and $\ell(Q_i) < \gamma_0 \ell(Q_{i+1})$ for every $i$, where we fix the value of the constant $\gamma_0$ later.
We set
\begin{align*}
  W_1 \coloneqq T_{Q_1}, \ \ \ \ \ \ W_k \coloneqq T_{Q_k} \setminus T_{Q_{k-1}}
\end{align*}
and
\begin{align*}
  \varphi_k \coloneqq 1_{W_k} \varphi_{Q_k}, \ \ \ \ \ \varphi \coloneqq \sum_k \varphi_k.
\end{align*}
Here $\varphi_{Q_k}$ is the function constructed in Section \ref{section:construction_local} for the cube $Q_k$.
The sets $W_k$ cover the whole space $\Omega$ and since $T_{Q_i} \subset T_{Q_{i+1}}$ for every $i$, they are also pairwise disjoint.
Let us consider the pointwise bound for $N_*(u-\varphi)$. Fix $x \in E$ and let $Q_m$ be the smallest of the previously chosen 
cubes such that $x \in Q_m$. Now, if $\Gamma(x) \cap T_{Q_j} = \emptyset$ for every $j = 1,2,\ldots,m-1$, then the pointwise 
bound follows directly from part i) of Lemma \ref{lemma:local_pointwise_bounds}. Suppose then that there exists a point 
$Y \in \Gamma(x) \cap T_{Q_j}$ for some $j < m$. We may assume that $Y \notin T_{Q_i}$ for all $i < j$. 
By the structure of the sets, there exist now cubes $P_1 \subset Q_m$ and $P_2 \subset Q_j$ such that 
$\ell(P_1) \approx \ell(P_2)$, $\dist(P_1,P_2) \lesssim \ell(P_1)$, $Y \in U_{P_1} \cap U_{P_2}$ and 
$\varphi(Y) = \varphi|_{U_{P_2}}(Y)$. By the considerations in the proof of part i) of Lemma \ref{lemma:local_pointwise_bounds},
we know that $|u(Y) - \varphi(Y)| \le \eps M_\D(N_*u)(P_2)$. By the properties of $P_1$ and $P_2$, there exists a uniform constant 
$\beta_0$ such that $P_1 \subset \beta_0 \Delta_{Q}$ for any $Q \in \D(E)$ such that $Q \supseteq P_2$. In particular,
\begin{align*}
  \eps M_\D(N_*u)(P_2)
  &= \eps \sup_{Q \in \D(E), P_2 \subseteq Q} \fint_{Q} N_*u \, d\sigma \\
  &\lesssim \eps \sup_{Q \in \D(E), P_2 \subseteq Q} \fint_{\beta_0 \Delta_Q} N_*u \, d\sigma
  \le \eps M(N_* u)(x).
\end{align*}
Thus,
\begin{align*}
  N_*(u-\varphi)(x) &= \sup_{Y \in \Gamma(x)} |u(Y) - \varphi(Y)| \\
                    &= \sup_{k \in \N} \sup_{Y \in \Gamma(x) \cap W_k} |u(Y) - \varphi(Y)|
                    \lesssim \eps M_\D(N_* u)(x).
\end{align*}

Let us then prove the $\Cc_\D$ estimate. We fix a point $x \in E$ and a cube $Q \in \D(E)$ such that $x \in Q$ and split 
the proof to three different cases. Below, $\beta$ and $\alpha$ are uniform constants and $m$ is the smallest such number
that $T_Q \subset T_{Q_m}$.
\begin{enumerate}
  \item[1)] $T_Q \subset T_{Q_m}$ such that $T_Q \cap T_{Q_k} = \emptyset$ for every $k < m$. Now we simply have
            \begin{align*}
              \iint_{T_Q} |\nabla \varphi| = \iint_{T_Q} |\nabla \varphi_m| \lesssim \frac{1}{\eps^2} \int_{\beta \Delta_Q} N_*^{\alpha} u \, d\sigma
            \end{align*}
            by Proposition \ref{proposition:local_gradient_bound}.
  
  \item[2)] $T_Q \subset T_{Q_m}$ and $Q_k \subset Q$ for every $k < m$. Take any $\overrightarrow{\Psi} \in C_0^1(T_Q)$ with $\|\overrightarrow{\Psi}\|_{L^\infty} \le 1$. We get
            \begin{align*}
              \iint_{T_Q} \varphi \, \text{div} \overrightarrow{\Psi}
               &=\iint_{T_Q \setminus T_{Q_{m-1}}} \varphi_m \, \text{div} \overrightarrow{\Psi} 
                 + \sum_{i=1}^{m-2} \iint_{T_{Q_{m-i}} \setminus T_{Q_{m-(i+1)}}} \varphi_{m-i} \, \text{div} \overrightarrow{\Psi}
                 + \iint_{T_{Q_1}} \varphi_1 \, \text{div} \overrightarrow{\Psi} \\
                 &\lesssim \frac{1}{\eps^2} \int_{\beta \Delta_Q} N_*^\alpha u \, d\sigma + \sum_{i=1}^{k-1} \frac{1}{\eps^2} \int_{\beta \Delta_{Q_i}} N_*^\alpha u \, d\sigma
            \end{align*}
            by Remark \ref{remark:modified_local_bounds}. We note that the balls $\beta \Delta_{Q_i}$ form an increasing sequence with respect to 
            inclusion. If we choose the constant $\gamma_0$ to be large enough, the balls $\beta \Delta_{Q_i}$ satisfy a Carleson packing condition independent of $m$. 
            Thus, for a large enough $\gamma_0$, we get
            \begin{align*}
              \frac{1}{\eps^2} \int_{\beta \Delta_Q} N_*^\alpha u \, d\sigma
                                                      + \sum_{i=1}^{k-1} \frac{1}{\eps^2} \int_{\beta \Delta_{Q_i}} N_*^\alpha u \, d\sigma 
              \lesssim \frac{1}{\eps^2} \int_{\beta \Delta_Q} M_\D(N_*^\alpha u) \, d\sigma.
            \end{align*}
            by a simple dyadic covering argument and the discrete Carleson embedding theorem (Theorem \ref{theorem:carleson_embedding}).
  
  \item[3)] $T_Q \subset T_{Q_m}$, $Q_k \not\subset Q$ for every $k < m$ and $T_Q \cap T_{Q_{m-1}} \neq \emptyset$. Without loss 
            of generality, we may assume that $\ell(Q) \approx \ell(Q_{m-1})$. Take any $\overrightarrow{\Psi} \in C_0^1(T_Q)$ with $\|\overrightarrow{\Psi}\|_{L^\infty} \le 1$. We get
            \begin{align*}
              \iint_{T_Q} \varphi \, \text{div} \overrightarrow{\Psi}
               &=\iint_{T_Q \setminus T_{Q_{m-1}}} \varphi_m \, \text{div} \overrightarrow{\Psi} 
                 + \sum_{i=1}^{m-2} \iint_{(T_Q \cap T_{Q_{m-i}}) \setminus T_{Q_{m-(i+1)}}} \varphi_{m-i} \, \text{div} \overrightarrow{\Psi} \\
                 &\qquad + \iint_{T_Q \cap T_{Q_1}} \varphi_1 \, \text{div} \overrightarrow{\Psi} \\
                 &\lesssim \frac{1}{\eps^2} \int_{\beta \Delta_Q} N_*^\alpha u \, d\sigma + \sum_{i=1}^{k-1} \frac{1}{\eps^2} \int_{\beta \Delta_{Q_i}} N_*^\alpha u \, d\sigma
            \end{align*}
            by Remark \ref{remark:modified_local_bounds}. Again, if we choose the constant $\gamma_0$ to be large enough, 
            we get
            \begin{equation*}
              \begin{split}
                \frac{1}{\eps^2} \int_{\beta \Delta_Q} N_*^\alpha u \, d\sigma + \sum_{i=1}^{k-1} \frac{1}{\eps^2} \int_{\beta \Delta_{Q_i}} N_*^\alpha u \, d\sigma
                \lesssim \frac{1}{\eps^2} \int_{\beta \Delta_Q} M_\D(N_*^\alpha u) \, d\sigma
              \end{split}
            \end{equation*}
            by a simple dyadic covering argument and the discrete Carleson embedding theorem (Theorem \ref{theorem:carleson_embedding}).
\end{enumerate}
Since $\sigma(Q) \approx \sigma(\beta \Delta_Q)$, combining the three cases gives us
\begin{align*}
  \frac{1}{\sigma(Q)} \iint_{T_Q} |\nabla \varphi| &\lesssim \frac{1}{\eps^2} \frac{1}{\sigma(Q)} \int_{\beta \Delta_Q} M_\D(N_*^\alpha u) \, d\sigma 
                                                   \lesssim \frac{1}{\eps^2} M(M_\D(N_*^\alpha u))(x)
\end{align*}
for almost every $x \in Q$. This completes the proof of Theorem \ref{thm:main_result_pointwise}.

\appendix

\section{Discrete Carleson embedding theorem}
\label{appendix:embedding_theorem}

For the convenience of the reader, we prove here the version of the Carleson embedding theorem that we used in Section \ref{section:E_unbounded}.

\begin{theorem}
  \label{theorem:carleson_embedding}
  Suppose that $\mu$ is a locally finite doubling Borel measure in a (quasi)metric space $X$ satisfying $\mu(B(x,r)) > 0$ for any $r > 0$ 
  and $\D$ is a dyadic system in $X$.
  Let $f \ge 0$ be a locally integrable function. If $\Ac \subset \D$ is a collection that satisfies a 
  Carleson packing condition with a constant $\Lambda \ge 1$, then
  \begin{align*}
    \sum_{Q \in \Ac, Q \subset Q_0} \int_Q f \, d\mu \le \Lambda \int_{Q_0} M_\D f \, d\mu
  \end{align*}
  for any $Q_0 \in \D$.
\end{theorem}

\begin{proof}
  For every $m \in \Z$, we define the averaging operator $\Tc_m$ by setting
  \begin{align*}
    \Tc_mf(x) = \sum_{\substack{Q \in \D \\ \ell(Q) \coloneqq 2^{-m}}} 1_{Q}(x) \fint_{Q} f \, d\mu,
  \end{align*}
  and we define the measure $\nu$ by setting
  \begin{align*}
    d\nu(x,m) = \left(\sum_{Q \in \Ac, \ell(Q) = 2^{-m}} 1_{Q}(x) \right) d\mu(x).
  \end{align*}
  Now we have
  \begin{align*}
    \sum_{Q \in \Ac, Q \subset Q_0} \int_Q f \, d\mu &= \sum_{Q \in \Ac, Q \subset Q_0} \mu(Q) \fint_Q f \, d\mu \\
                                                     &= \sum_{m: \, 2^{-m} \le \ell(Q_0)} \sum_{\substack{Q \in \Ac \\ \ell(Q) = 2^{-m}}} \int_{Q_0} 1_{Q} \left(  \fint_{Q} f \right) \, d\mu \\
                                                     &= \sum_{m: \, 2^{-m} \le \ell(Q_0)} \int_{Q_0} \Tc_mf(x) \, d\nu(x,m) \\
                                                     &= \int_0^\infty \nu(E_\lambda^*) \, d\lambda,
  \end{align*}
  where $E_\lambda^* \coloneqq \{(x,m) \colon x \in Q_0, 2^{-m} \le \ell(Q_0), \Tc_mf(x) > \lambda\}$. Thus, to prove the claim, we 
  only need to show that $\nu(E_\lambda^*) \le \Lambda \mu(E_\lambda)$, where $E_\lambda \coloneqq \{x \in Q_0 \colon \sup_m \Tc_mf(x) > \lambda\}$.
  If $\mu(E_\lambda) = \infty$, the claim is trivial. Thus, we may assume that $\mu(E_\lambda) < \infty$.
  
  We notice that if $x \in E_\lambda$, then there exists a subcube $Q' \subset Q_0$ such that $x \in Q'$ and 
  $\fint_{Q'} f \, d\mu > \lambda$. By the definition of $\Tc_m$, we also have $y \in E_\lambda$ for every $y \in Q'$.
  In particular, we have maximal disjoint subcubes $R_j \subset Q_0$ such that $E_\lambda = \bigcup_j R_j$.
  We further observe the following two things:
  \begin{enumerate}
    \item[$\bullet$] If $x \in Q_0 \setminus \bigcup_j R_j$, then by the maximality of the cubes $R_j$ we have $\sup_m \Tc_mf(x) \le \lambda$.
  
    \item[$\bullet$] If $x \in Q \subset Q_0$ and $\Tc_mf(x) > \lambda$ for some $m$ such that $2^{-m} > \ell(Q)$, then 
                     there exists a cube $\widetilde{Q} \supsetneq Q$ such that $\fint_{\widetilde{Q}} f \, d\mu > \lambda$.
                     In particular, $Q \subset E_\lambda$ but $Q$ is not a maximal cube.
  \end{enumerate}
  Based on these observations, we have
  \begin{align*}
    E_\lambda^* \subset \bigcup_j R_j \times \{m \colon 2^{-m} \le \ell(R_j)\}.
  \end{align*}
  By the Carleson packing condition, we get
  \begin{align*}
    \nu(R_j \times \{m \colon 2^{-m} \le \ell(R_j)\})
    = \sum_{m: \, 2^{-m} \le \ell(R_j)} \sum_{\substack{Q' \subset R_j, Q' \in \Ac \\ \ell(Q') = 2^{-m}}} \mu(Q')
    \le \Lambda \mu(R_j)
  \end{align*}
  for every $j$. In particular, since the cubes $R_j$ are disjoint, we get
  \begin{align*}
    \nu(E_\lambda^*) \le \sum_j \nu(R_j \times \{m \colon 2^{-m} \le \ell(R_j)\}) \le \sum_j \Lambda \mu(P_j) = \Lambda \mu(E_\lambda),
  \end{align*}
  which completes the proof.
\end{proof}

\bibliographystyle{alpha}
\bibliography{approximability}

\end{document}